\documentclass[11pt, a4paper]{amsart}

\usepackage[foot]{amsaddr}  

\usepackage[margin=2.54cm]{geometry}
\usepackage{fancyhdr}

\usepackage{amsmath, amssymb, amsthm}
\usepackage{array}
\usepackage{latexsym}
\usepackage{nicematrix}
\usepackage{bigstrut}
\usepackage{enumerate}
\usepackage{paralist}

\usepackage{mathrsfs}

\usepackage{float}   
\usepackage{subcaption}  
\usepackage{pgf,tikz}
\usetikzlibrary{decorations.pathreplacing,calligraphy}
\usepackage{graphicx}
\usepackage[all]{xy}
\usetikzlibrary{arrows}
\usetikzlibrary{arrows.meta}

\usepackage{url}
\usepackage[colorlinks=true,citecolor=cyan,backref=page]{hyperref}
\usepackage[capitalise]{cleveref}

\usepackage{fdsymbol}
\usepackage{todonotes}
\setuptodonotes{size=\tiny}

\makeatletter
\define@key{todonotes}{AS}[]{%
	\setkeys{todonotes}{color=blue!50}}%

\makeatother

\renewcommand{\int}{{\mathord{\hookrightarrow}}}

\newcommand{\onto}{%
  \mathrel{\tikz[baseline=-0.5ex]{
    \draw[->>,line width=0.2pt] (0,0) -- (0.4,0);
  }}%
}

\newcommand{\into}{%
  \mathrel{\tikz[baseline=-0.5ex]{
    \draw[right hook->,line width=0.2pt] (0,0) -- (0.4,0);
  }}%
}

\newcommand{\twoheadrighthookarrow}{%
  \mathrel{\tikz[baseline=-0.5ex]{
\draw[right hook->>,line width=0.3pt] (0,0) -- (0.4,0);
  }}%
}

\newtheorem{theorem}{Theorem}[section]
\newtheorem{corollary}[theorem]{Corollary}
\newtheorem{proposition}[theorem]{Proposition}
\newtheorem{lemma}[theorem]{Lemma}

\newtheorem{question}[theorem]{Question}

\theoremstyle{definition}
\newtheorem{definition}[theorem]{Definition}
\newtheorem{remark}[theorem]{Remark}

\title{Extremality in semidistributive lattices}

\author[A.~Segovia]{Adrien~Segovia}
\address{LACIM, Universit\'e du Qu\'ebec \`a Montr\'eal, Canada}
\email{segovia.adrien@courrier.uqam.ca}

\begin{document}

	\begin{abstract}
We establish several independent results concerning extremal, left modular, congruence uniform, and semidistributive lattices. An equivalent characterization of left modular lattices is obtained in terms of edge-labellings, together with necessary and sufficient conditions on the doubling steps in the construction of congruence normal lattices that ensure left modularity or extremality. We prove that a congruence uniform lattice is shellable if and only if it is extremal. We answer a question of Barnard by constructing a counterexample showing that an induced subcomplex of a canonical join complex need not itself be such a complex. Finally, we show that the order dimension of a semidistributive extremal lattice equals the chromatic number of the complement of its Galois graph, generalizing a theorem of Dilworth for distributive lattices. As an application, we determine the dimensions of generalizations of the Hochschild lattice, of the parabolic Tamari lattice, and of some lattices of torsion classes.
\end{abstract}
	
	\maketitle
	
	
	\tableofcontents

\section{Introduction}

\textit{In this article all posets and lattices are assumed to be finite. The definitions will be given in this context. All the notions will be defined in Section \ref{sec:Backgroundext}.}

A lattice is \emph{extremal} if the maximum size of a chain in the lattice corresponds to the number of join-irreducible and to the number of meet-irreducible elements, which are the elements that respectively cover or are covered by a unique element. It is known that the distributive lattices are the extremal and graded lattices \cite{MarkowskyExtremal}. We will be mainly interested in the \emph{semidistributive} lattices, which are a generalization of the distributive lattices, and to a subclass of them called \emph{congruence uniform} lattices. These latter are the lattices that can be obtained from the one element lattice by doubling intervals using the doubling construction of A. Day \cite{Daydoubling}. By doubling convex subsets, we obtain the \emph{congruence normal lattices}, which generalize the congruence uniform lattices but are not semidistributive, unless they are congruence uniform. Some examples of congruence uniform lattices are the distributive lattices, the Tamari lattice and most of its generalizations, and the weak order on a finite Coxeter group. With the exception of the weak order which is not an extremal lattice, the others are semidistributive and extremal and we will give more examples of such lattices in Section \ref{sec:applicationdimension}.

An element of a lattice is left modular if it cannot be obtained as the small side of a pentagonal sublattice. A lattice is \emph{left modular} if it contains a maximal chain made of left modular elements \cite{blass1997mobius,LIUSagan}. We call such a chain a left modular chain. As a distributive lattice does not have pentagonal sublattices, all its elements are left modular, in particular it is a left modular lattice. A poset is \emph{shellable} when its order complex is a shellable  simplicial complex, which is a nice topological property that is often difficult to establish. In fact, determining whether a given poset is shellable is NP-complete \cite{goaoc2019shellability}. Proofs of shellability of lattices often rely on finding a particular edge-labelling called an $EL$-labelling \cite{bjorner1983lexicographically}. Such a labelling can be obtained in a left modular lattice \cite{LiuLeftmodular}, thus a left modular lattice is shellable.

Let $L$ be a lattice. Choosing a chain $\psi : \hat{0}=x_0< \dots < x_k=\hat{1}$ that contains the minimum $\hat{0}$ and maximum $\hat{1}$ of $L$, we get a numbering of its join-irreducible and meet-irreducible elements, that enables us to define four edge-labellings $\gamma_{1,\psi},\gamma_{1',\psi},\gamma_{2,\psi},\gamma_{2',\psi}$ of $L$ (see Definition \ref{deflabelling}). We will prove:

\begin{theorem}  \label{thm:labellingsintro}
	For any lattice $L$, $\gamma_{1,\psi}=\gamma_{1',\psi}= \gamma_{2,\psi}=\gamma_{2',\psi}$ if and only if for all $i$, $x_i$ is left modular.
\end{theorem}

This gives us a way to prove that lattices are left modular using edge-labellings. A first application, not present in this paper, is the proof of the left modularity of the $(P,\phi)$-Tamari lattices defined recently by the author \cite{segovia2025pphitamarilattices}.
Another application is a new short proof (Theorem \ref{thmlabelling}) of the fact that every semidistributive extremal lattice is left modular \cite[Theorem 1.4]{Thomas_2019}.
The converse statement, that any semidistributive left modular lattice is extremal, was proved in \cite[Theorem 3.2]{M_hle_2023}. Thus for semidistributive lattices, extremality and left modularity are equivalent properties. We were interested in understanding this equivalence in the case of the congruence uniform lattices. For this, we obtain necessary and sufficient conditions on the doubling steps in the construction of a congruence normal lattice that ensure extremality or left modularity. To state these results we need the following two definitions. The \emph{spine} of a lattice $L$ is the subset of elements that lie on some maximum length chain. We define the \emph{heart} of a subset $C$ of $L$ to be the subset of elements that are less than all the maximal elements of $C$ and bigger than all its minimal elements.

\begin{theorem} \label{thm:extremalleftmodularintro}
	A congruence uniform lattice is extremal if and only if at each doubling step in its construction, the interval to be doubled intersects the spine of the lattice we are doubling. A congruence normal lattice is left modular if and only if at each doubling step in its construction, the heart of the convex subset to be doubled intersects a left modular chain of the lattice we are doubling.
\end{theorem}

From Theorem \ref{thm:extremalleftmodularintro} we recover the equivalence between extremality and left modularity for congruence uniform lattices (Corollary \ref{corocongruenceuniformExtremal}), but we will be more precise and prove more general results in Section \ref{sec:congNormal}.

As any lattice can be seen as an interval in an extremal lattice \cite[Theorem 14]{MarkowskyExtremal} and shellability is preserved under taking intervals \cite[Proposition 4.2]{bjorner1980shellable}, being extremal is far from being equivalent to being shellable for general lattices.
However, building on the above result regarding the extremality of a congruence uniform lattice, we obtain: 

\begin{theorem}  \label{thm:principal}
	A congruence uniform lattice is shellable if and only if it is extremal.
\end{theorem}

We said that establishing shellability is often difficult, but Theorem \ref{thm:principal} proves that for congruence uniform lattices this is easy as extremality can be obtained in polynomial time for any lattice. Furthermore, with known results explained earlier, it follows from Theorem \ref{thm:principal}:

\begin{corollary} \label{thm:mainthmintro}
	Let $L$ be a congruence uniform lattice. The following properties are equivalent:
	\begin{itemize}
		\item[$(i)$] $L$ is extremal.
		\item[$(ii)$] $L$ is left modular.
		\item[$(iii)$] $L$ is $EL$-shellable.
		\item[$(iv)$] $L$ is shellable.
	\end{itemize}       
\end{corollary}

Part of what was known, that an extremal semidistributive lattice is left modular, has led to recent results on shellability of lattices \cite{fang2024parabolic,defant2024operahedron,luo2024lattice}. What we add with Theorem \ref{thm:principal} is a way to prove that certain lattices are not shellable, as now we know that any congruence uniform lattice that is not extremal is not shellable.

In a semidistributive lattice, each element can be obtained canonically as the join of some join-irreducible elements. These subsets of join-irreducible elements are called the canonical join representations, and the set of all of them form a simplicial complex called the \emph{canonical join complex} \cite{barnard2016canonical}. Different structural properties of these simplicial complexes were proved, but open questions remain. In Section \ref{sec:canonicaljoingraph}, we give a counterexample to \cite[Question 2.5.5]{Barnardthesis}, by exhibiting a join canonical complex that has an induced subcomplex which is not itself a canonical join complex. This counterexample is produced using the fundamental theorem of finite semidistributive lattices \cite{reading2021fundamental}.

Finally, we are interested in the (order) dimension of semidistributive extremal lattices. Let $(\mathbb{R}^n,\leq)$ be the poset defined by $(x_1,\dots,x_n)\leq (y_1,\dots,y_n)$ if $x_i\leq y_i$ for every $i\in [n]$. The \emph{dimension} of a poset $P$ is the smallest integer $n$ such that $P$ is isomorphic to a subposet of $(\mathbb{R}^n,\leq)$ (\cite{DushnikMiller}). This notion was extensively studied over the years (see the book \cite{trotter2002combinatorics}), but it is difficult to obtain the dimension in general, as it is even an NP-complete problem to determine if a given poset is of dimension $n$, for any $n\geq 3$. In the past, this problem was linked to the coloring of graphs and hypergraphs \cite{YannakakisNPcomplet,furedi1992interval,felsner2000dimension,READING2003265}, but often to obtain bounds on the dimension.

The \emph{Galois graph} of an extremal lattice is a directed graph on its join-irreducible elements. It is acyclic and from it we recover the lattice as the poset of inclusion on its maximal orthogonal pairs \cite{MarkowskyExtremal,reading2021fundamental,Thomas_2019,thomas2018independence}. 

Building on results of \cite{reading2002order} which relate the dimension to cover sets of a simplicial complex called the critical pairs simplicial complex, and results about the Galois graphs of extremal lattices, we obtain the following:

\begin{theorem} \label{thm:dimresultintro}
	The dimension of a semidistributive extremal lattice is the chromatic number of the complement of its Galois graph.
\end{theorem}

This generalizes a result of Dilworth that the dimension of a distributive lattice is the maximum size of an antichain of its subposet of join-irreducible elements (see Corollary \ref{cor:dilworth}).

The complement of the Galois graph is often explicitly described for semidistributive lattices, as it corresponds to the \emph{canonical join graph}, which is the $1$-skeleton of its canonical join complex \cite{barnard2016canonical,Thomas_2019}. Thus, using Theorem \ref{thm:dimresultintro} we are able to give the dimension of different families of semidistributive extremal lattices, and we expect that this result could be applied effectively to obtain the dimension of other families of such lattices. We will focus on some recent families of semidistributive extremal lattices appearing in algebraic combinatorics, mainly generalizations of the Hochschild lattices and of the Tamari lattice. For the latter, this includes results about the dimension of some lattices of torsion classes. In particular, we prove that the lattice of torsion classes of a gentle tree on $n$ vertices has dimension $n$ (Proposition \ref{prop:dimgentletree}).

We start with background on the relevant poset, lattice and representation theory in Section \ref{sec:Backgroundext}. Then Section \ref{sec:mainresults} gives our main results. In Section \ref{sec:EdgelabellingLeftModular} we prove Theorem \ref{thm:labellingsintro}. In Section \ref{sec:congNormal}, we prove Theorem \ref{thm:extremalleftmodularintro} and more results relating the doubling operation to extremality and left modularity. In Section \ref{sec:shellabilityresult}, we focus on shellability and prove Theorem \ref{thm:principal}. Then in Section \ref{sec:canonicaljoingraph} we give a counterexample to \cite[Question 2.5.5]{Barnardthesis}. We then focus on the dimension of semidistributive extremal lattices in Section \ref{sec:dimensionresult}, proving Theorem \ref{thm:dimresultintro}. Finally, we give applications of Theorem \ref{thm:dimresultintro} in Section \ref{sec:applicationdimension}, and finish by giving a list of open questions in Section \ref{sec:openquestions}.

\textit{Some of the results from Sections \ref{sec:EdgelabellingLeftModular} and \ref{sec:congNormal} appeared, without proofs, in an extended abstract accepted to FPSAC 2025 \cite{segovia2025pphitamarilattices}.}

\medskip
\noindent{\bf Acknowledgments} 

I am very grateful to my advisors Samuele Giraudo and Hugh Thomas. I also want to thank Antoine Abram for helpful discussions, Félix Reutenauer who used a supercomputer to run my program to obtain Proposition \ref{prop:computercambrian}, and Clément Chenevière for pointing out some known results on the Cambrian lattices. I was supported by NSERC Discovery Grants RGPIN-2022-03960 and RGPIN-2024-04465.

\section{Background} \label{sec:Backgroundext}

\subsection{Vocabulary on posets and lattices}
\label{sec:vocabulary}

We denote by $|E|$ the cardinality of a set $E$. For a positive integer $n$, we denote $\{1,2,\dots,n\}$ by $[n]$.
The terms \emph{smaller}, \emph{bigger}, \emph{below} and \emph{above} will always refer to weak relations.

If a partially ordered set (poset) $(P,\leq)$ has a minimum, it will always be denoted $\hat{0}$, and $\hat{1}$ for the maximum. The cover relations of $P$ are denoted $x\lessdot y$ and we say that $y$ \emph{covers} $x$ or $x$ \emph{is covered by} $y$. They form the set $E(P)$ of the edges of its Hasse diagram, which is drawn with smaller elements at the bottom. The interval of $P$ between $x$ and $y$ is denoted by $[x,y]$. 
A subset $C$ of $P$ is \emph{convex} if for all $x$ and $y$ in $C$, we have $[x,y]\subseteq C$.   An \emph{order ideal} of a poset $P$ is a subset $I$ such that for all $x,y\in P$, if $y\leq x$ and $x\in I$, then $y\in I$. The order ideal generated by a subset $C$ is denoted by $I_P(C):=\{y\in P\mid \exists x\in C,\,y\leq x\}$. If $C=\{x\}$, we write $I_P(x)$.
Dually, an \emph{order filter} is a subset $F$ such that for all $x,y\in P$, if $y\geq x$ and $x\in F$, then $y\in F$. The order filter generated by a subset $C$ is denoted by $F_P(C):=\{y\in P\mid \exists x\in C,\,y\geq x\}$. If $C=\{x\}$, we write $F_P(x)$.
The \emph{chains} are the totally ordered subsets of $P$. 

A chain $F$ of $n+1$ elements is said to have length $n$, which is denoted by $\ell(F)$. If we cannot add elements to a chain, it is called a \emph{maximal chain}. The \emph{length} of a poset, denoted $\ell(P)$, is the maximum length of a chain in $P$. The chains of length $\ell(P)$ are called \emph{longest chains}. The \emph{spine} of a poset is the subset of the elements that lie on some longest chain. 
If $P$ and $Q$ are two posets, then the \emph{direct product} $P\times Q$ is the poset on the Cartesian product $P\times Q$ defined by $(x,y)\leq (x',y')$ if $x\leq x'$ and $y\leq y'$.

A \emph{lattice} $L$ is a poset such that any pair of elements $\{x,y\}$ admits a least upper bound, called the \emph{join} and written $x\vee y$, and a greatest lower bound, called the \emph{meet} and written $x\wedge y$.  
A \emph{join-irreducible} $j\in L$ is an element that covers a unique element, denoted $j_*$, and a \emph{meet-irreducible} $m\in L$ is one that is covered by a unique element, denoted $m^*$. 
The sets of these elements are respectively denoted $\mathrm{JIrr}(L)$ and $\mathrm{MIrr}(L)$. An \emph{edge-labelling} of $L$ is a map $\gamma: E(L) \rightarrow P$ where $P$ is a poset.
We will say that $x\in L$ is \emph{incomparable} to a subset $Y\subseteq L$ if $x$ is incomparable to any element of $Y$.

\subsection{Extremal lattices}
\label{sec:extremallattices}

It is well known that in a lattice, an element is always the join of the join-irreducibles below it.

\begin{lemma}
	\label{joinextrem}
	For any lattice $L$ we have $\ell(L)\leq \min\big(|\mathrm{JIrr}(L)|,|\mathrm{MIrr}(L)|\big)$.  
\end{lemma}

\begin{proof}
	Denote $k=\ell(L)$. Let $\hat{0}=c_0\lessdot c_1\lessdot \dots \lessdot c_{k}=\hat{1}$ be a maximal chain of length $k$. For any $i\in [k-1]$, there is always a join-irreducible that is below $c_{i+1}$ but not $c_{i}$. If it was not the case, then all the join-irreducibles below $c_{i+1}$ would be below $c_i$, thus $c_{i+1}$ which is the join of them would be below $c_i$, which is absurd. It follows that  $k\leq |\mathrm{JIrr}(L)|$. A dual argument gives $k\leq |\mathrm{MIrr}(L)|$.
\end{proof} 

The previous result motivated the following definition by G. Markowsky \cite{MarkowskyExtremal}:

\begin{definition}
	A lattice $L$ is \emph{join-extremal} if $\ell(L)=|\mathrm{JIrr}(L)|$, \emph{meet-extremal} if $\ell(L)=~|\mathrm{MIrr}(L)|$ and \emph{extremal} if it is both join-extremal and meet-extremal. 
\end{definition}

Two notable results on these lattices are that the graded extremal lattices are the distributive lattices \cite[Theorem $17$]{MarkowskyExtremal}, and any lattice is isomorphic to an interval of an extremal lattice \cite[Theorem $14$]{MarkowskyExtremal}. G. Markowsky also proved a representation theorem of extremal lattices, that we now explain with recent terminology (see \cite[Section 2.3]{Thomas_2019}).
For an extremal lattice $L$, the choice of any longest chain $\psi: \,\hat{0}=x_0 \lessdot \dots \lessdot x_n=\hat{1}$ gives a numbering of the join and meet-irreducibles $j_1,j_2,\dots,j_n$ and $m_1,m_2,\dots,m_n$ such that 
$x_i=j_1\vee \dots \vee j_i = m_{i+1}\wedge \dots \wedge m_n$. 

\begin{definition}[\cite{Thomas_2019}]
\label{def:galoisgraph}
	Let $L$ be an extremal lattice.
	We define a directed graph $G(L)$, called the \emph{Galois graph}, whose vertices are $\mathrm{JIrr}(L)$ and with an edge $j_i\rightarrow j_k$ if $i\neq k$ and $j_i\not\leq m_k$.
\end{definition}

We defined the Galois graph $G(L)$ on the set $\mathrm{JIrr}(L)$ as that will be the way we see it in Section \ref{sec:dimensionresult}, but with our numbering of join and meet-irreducibles, it can be equivalently seen as the graph on $[n]$ with directed edges $i\rightarrow k$ if $i\neq k$ and $j_i\not\leq m_k$.

\begin{proposition}[{\cite[Theorem 11]{MarkowskyExtremal}}] \label{prop:acyextremal}
	For any extremal lattice $L$, the Galois graph $G(L)$ is acyclic.
\end{proposition}

Let $G$ be a directed graph on $[n]$ without multiple edges and such that whenever $i\rightarrow k$ we have $i>k$. A \emph{maximal orthogonal pair} of $G$ is a pair $(X,Y)$ of disjoint subsets of $[n]$ with no directed edges from $X$ to $Y$, and which is maximal for this condition. With the partial order $(X,Y) \leq (X',Y')$ if $X\subseteq X'$, the maximal orthogonal pairs form a lattice denoted by $L(G)$ (see Figure \ref{fig:hoch3}).

\begin{theorem}[\cite{MarkowskyExtremal}]
	\label{thm:extremalrepthm}
	Every finite extremal lattice $L$ is isomorphic to $L(G(L))$. Conversely, if $G$ is a directed graph on $[n]$ without multiple edges and such that whenever $i\rightarrow k$ we have $i>k$, then $L(G)$ is an extremal lattice.
\end{theorem}

This result will enable us in Section \ref{sec:applicationdimension} to define quickly an extremal lattice by giving its Galois graph.

\subsection{Semidistributive lattices}
\label{sec:semidistributivelattices}

\begin{definition}
	A lattice $L$ is \emph{join-semidistributive} if for all $x,y,z\in L$, we have $x\vee y=x\vee z \Longrightarrow x\vee (y\wedge z)=x\vee y$. It is \emph{meet-semidistributive} if for all $x,y,z\in L$, we have $x\wedge y=x\wedge z \Longrightarrow x\wedge (y\vee z)=x\wedge y$. It is \emph{semidistributive} if it is both join-semidistributive and meet-semidistributive.   
\end{definition}

The following is well-known:

\begin{lemma}[{\cite[Section 5]{freese1995free}}]
	\label{lemEquiJSD}
	A lattice $L$ is join-semidistributive if and only if for all covers $b\lessdot c$, the set $I_L(c)\setminus I_L(b)$ has a minimum element, denoted $\gamma_{J}(b\lessdot c)$. It is meet-semidistributive if and only if for all covers $b\lessdot c$, the set $F_L(b)\setminus F_L(c)$ has a maximum element, denoted $\gamma_{M}(b\lessdot c)$. In these cases, the minimum is join-irreducible and the maximum meet-irreducible.
\end{lemma}

Thus by Lemma \ref{lemEquiJSD}, if $L$ is semidistributive we have two edge-labellings $\gamma_J$ and $\gamma_M$. In this case, the map $\kappa: j\in \mathrm{JIrr}(L) \mapsto \gamma_M(j_*\lessdot j)\in \mathrm{MIrr}(L)$ is a bijection \cite[Theorem 2.54]{freese1995free}.

\begin{proposition}[{\cite{barnard2016canonical} and \cite[Proposition 6.2]{Thomas_2019}}]
	\label{prop:kappaSD}
	For all $b\lessdot c$ in a semidistributive lattice $L$, we have $\kappa \big(\gamma_J(b\lessdot c)\big) = \gamma_M(b\lessdot c)$.  
\end{proposition}

Recall as explained earlier that if $L$ is extremal, choosing a longest chain of $L$ gives a numbering of the join-irreducible elements $j_1,\dots, j_n$ and the meet-irreducible elements $m_1,\dots,m_n$. 
The next statement follows from Proposition \ref{prop:kappaSD}: 

\begin{proposition}
	\label{prop:kappaextremalSD}
	Let $L$ be a semidistributive extremal lattice. We have $\kappa(j_i)=m_i$ for all $i\in [n]$. Thus the Galois graph $G(L)$ is the directed graph whose vertices are $\mathrm{JIrr}(L)$ and there is an edge $i\rightarrow j$ if $i\neq j$ and $i\not\leq \kappa(j)$ (we use the same convention as explained after Definition \ref{def:galoisgraph}; the join-irreducible $j_i$ being replaced by the integer $i$ for all $i\in [n]$).
\end{proposition}

Let $L$ be a semidistributive lattice. For any $x\in L$, we denote $D(x):=\{\gamma_J (y\lessdot x) \mid y\in~L,\, y\lessdot x\}$ and call it the \emph{downward label set} of $x$. The canonical join complex is often defined using the notion of canonical join representations, but the following definition is equivalent:

\begin{definition}[\cite{barnard2016canonical}]
The \emph{canonical join complex} of a semidistributive lattice is the set of downward label sets.  
\end{definition}

It is more generally defined for join-semidistributive lattices, but in this article we will always assume  that the canonical join complexes come from a semidistributive lattice.
The name is justified by the fact that it is a simplicial complex (\cite[Proposition $2.2$]{readingcanonical}). The \emph{clique complex} of a graph $G$ is the simplicial complex whose faces are the cliques of $G$.
The canonical join complex is flag (\cite[Theorem 1]{barnard2016canonical}), which by definition means that it is the clique complex of its 1-skeleton. This 1-skeleton is called the \emph{canonical join graph} of the lattice. 

The \emph{complement} of a (directed) graph $H$, denoted by $\overline{H}$, is the simple undirected graph on the same vertices where we have an edge between two vertices if there are no (directed) edges between the two in $H$. Thus, the cliques of $\overline{H}$ are the independence sets of $H$.

\begin{lemma}[{\cite[Corollary 6.7]{Thomas_2019}}]
\label{lem:canonicaljoingraphiscompgalois}
The canonical join graph of a semidistributive extremal lattice is the complement of its Galois graph.    
\end{lemma}

\subsection{Left modular lattices}
\label{sec:leftmodularlattices}

The left modular lattices were first defined and studied by A. Blass and B. Sagan \cite{blass1997mobius}.

\begin{definition}
	An element $a\in L$ is \emph{left modular} if for all $b<c$ in $L$, we have $(b\vee a) \wedge c = b\vee (a\wedge c)$ (we gave an equivalent definition in the introduction, which is part of the next result).
	A maximal chain made of left modular elements is called a \emph{left modular chain}. The lattice $L$ is called left modular if there exists a left modular chain. 
\end{definition}

\begin{proposition}[\cite{LIUSagan}]
	\label{Propleftmodular}
	Let $a\in L$ be an element of a lattice $L$. The following statements are equivalent :
	\begin{itemize}
		\item[$(1)$] The element $a$ is left modular.
		\item[$(2)$] For all $b<c$, we have $a\vee b \neq a\vee c$ or $a\wedge b \neq a\wedge c$.
		\item[$(3)$] For all $b\lessdot c$, we have $a\vee b\neq a\vee c$ or $a\wedge b\neq a\wedge c$.
		\item[$(4)$] There is no pentagonal sublattice $N_5$ of $L$ such that $a$ is the middle element on the smaller maximal chain of $N_5$.
	\end{itemize}
\end{proposition}

If $b\lessdot c$ is such that $(3)$ of Proposition \ref{Propleftmodular} is not satisfied for $a$, we will say that $b\lessdot c$ is a \emph{counterexample to the left modularity} of $a$.

H. Thomas defined a lattice to be \emph{trim} if it is extremal and left modular \cite{thomas2005analoguedistributivityungradedlattices}. Trim lattices form a class of lattices, closed under taking intervals and quotients, that include in particular the Tamari lattices and more generally the finite Cambrian lattices. 
With N. Williams, they obtained additional results on them, including the fact that extremal semidistributive lattices are trim \cite[Theorem 1.4]{Thomas_2019}. We will give a new proof of this result in Section \ref{sec:EdgelabellingLeftModular}.

\subsection{Congruence normal lattices}
\label{sec:congruencenormallattices}

We will use the following doubling construction of A. Day \cite{Daydoubling}:

\begin{definition}
	\label{defDoublementDay} 
	Let $C$ be a convex subset of a lattice $L$. Let $\{0,1\}$ be the chain on two elements with order $0<1$. We define the \emph{doubling} $L[C]$ to be the subposet of $L\times \{0,1\}$ on the subset
	$$\Big(I_L(C) \times \{0\}\Big) \,\bigsqcup \,\Big[\Big( \big(L\setminus I_L(C) \big)\cup C\Big) \times \{1\} \Big].$$
	
	In fact $L[C]$ is a lattice with meet and join given for $(x,\epsilon),(y,\epsilon')\in L[C]$ by
	\begin{align*}
		(x,\epsilon) \vee (y,\epsilon') &=\left\{\begin{array}{cc}
			(x\vee y, \mathrm{max}(\epsilon,\epsilon')) & \text{if}\,\,x\vee y\in I_L(C)\\
			(x\vee y,1) & \text{otherwise,}
		\end{array}
		\right.   \\
		(x,\epsilon) \wedge (y,\epsilon') &=\left\{\begin{array}{cc}
			(x\wedge y, \mathrm{min}(\epsilon,\epsilon')) & \text{if}\,\,x\wedge y\in\left((L\setminus I_L(C))\bigcup C\right)\\
			(x\wedge y,0) & \text{otherwise.}
		\end{array}
		\right.   \\
	\end{align*}
\end{definition}

We have a surjective poset morphism $L[C]\rightarrow L$ that sends $(x,\epsilon)$ to $x$, and an injective poset morphism $\psi:L \rightarrow L[C]$ that sends $a\in L$ to $(a,0)$ if $a\in I_L(C)$ or to $(a,1)$ otherwise.
If $F$ is a maximal chain of $L[C]$, we say that $F$ \emph{goes through the doubling} of $C$ if there exists $c\in C$ such that $\{(c,0),(c,1)\} \subseteq F$. 
See Figure \ref{fig:doublingsLeftmodular} for examples of the doubling construction. 

From the definition of the doubling construction, the next two results follow.

\begin{lemma}[{\cite[Lemma 1]{CASPARD200471}}]
	\label{lemJIRR}
	The join-irreducible elements of $L[C]$ are of two types;  each $j\in \mathrm{JIrr}(L)$ gives $(j,\epsilon) \in \mathrm{JIrr}(L[C])$ with $\epsilon = 0$ if $j\in I_L(C)$ and $\epsilon =1$ otherwise, and we also get $(p,1)\in \mathrm{JIrr}(L[C])$ for any minimal element $p$ of $C$. 
	
	Similarly the meet-irreducible elements of $L[C]$ are of two types;  each $m\in \mathrm{MIrr}(L)$ gives $(m,\epsilon) \in \mathrm{MIrr}(L[C])$ with $\epsilon = 0$ if $m\in I_L(C)\setminus C$ and $\epsilon =1$ otherwise, and we also get $(q,0)\in \mathrm{MIrr}(L[C])$ for any maximal element $q$ of $C$. 
\end{lemma}

\begin{lemma} \label{lem:lengthdoublement}
	We have $\ell(L[C])=\ell(L)$ if $C$ does not intersect the spine of $L$, or $\ell(L[C])=\ell(L)+1$ otherwise.   
\end{lemma}

A subset $C$ is a \emph{lower pseudo-interval} if $C$ is a union of intervals sharing the same minimum element. It is an \emph{upper pseudo-interval} if it is a union of intervals sharing the same maximum element. A lattice $L$ is \emph{congruence normal} if it is obtained from the one element lattice by successive doublings of convex subsets. We write $L=E[C_1,C_2,\dots,C_n]$ for the lattice obtained from the one element lattice $E$ by doubling successively $C_1, C_2,\dots,C_n$, which means $E[C_1,\dots,C_{i+1}]=E[C_1,\dots,C_i][C_{i+1}]$ for all $i$ where $C_{i+1}$ is a convex subset of $E[C_1,\dots,C_i]$. We will always assume that $C_i\neq \emptyset$ for all $i$. If $C_i$ is a lower pseudo-interval for all $i$ then $L$ is \emph{join-congruence uniform}. If all the $C_i$ are upper pseudo-intervals then $L$ is \emph{meet-congruence uniform}. If all the $C_i$ are intervals, $L$ is called \emph{congruence uniform}\footnote{\emph{Join-congruence uniform} and \emph{meet-congruence uniform} lattices are often called respectively \emph{lower-bounded} and \emph{upper-bounded} in the literature, but it conflicts with the term \emph{bounded} poset which usually refers to a poset with a minimum and a maximum.}. 

\begin{lemma}[\cite{Daydoubling}] \label{lem:congisSD}
	A join-congruence uniform lattice is join-semidistributive. A meet-congruence uniform lattice is meet-semidistributive. A congruence uniform lattice is semidistributive.
\end{lemma}

\subsection{Shellability}
\label{sec:shellability}

We refer the reader to the papers of A. Björner and M. Wachs \cite{bjorner1983lexicographically,bjorner1996shellable,bjorner1997shellable} for details related to the following topological notions. The \emph{order complex} $\Delta(P)$ of a poset $P$ is the simplicial complex on vertex set $P$ whose faces are the chains of $P$. Thus the facets of $\Delta(L)$ are the maximal chains of $L$. The \emph{dimension} of a face $F$ is $\dim(F):=|F|-1$. For any face $F$ of a simplicial complex, we denote by $\langle F \rangle:=\{G \mid G\subseteq F\}$ the subcomplex generated by $F$. A simplicial complex is \emph{shellable} if there exists a linear order of its facets $F_1,F_2,\dots,F_k$ such that for any $1\leq j<k$, the complex 
\begin{align}
	\big(\bigcup_{i=1}^j \langle F_i \rangle\big) \bigcap \langle F_{j+1} \rangle  \label{eqdefshelling}   
\end{align}
is pure of dimension $\dim(F_{j+1})-1$. We call such a linear order on the facets a \emph{shelling} of $\Delta(L)$. Equivalently, $F_1,F_2,\dots,F_k$ is a shelling if and only if for all $1\leq i<j\leq k$, there exists $l<j$ such that $F_i\,\cap F_j \subseteq F_l\,\cap F_j$ and $|F_l\,\cap F_j|=|F_j|-1$. We will say that $L$ is shellable if $\Delta(L)$ is shellable. It is an easy fact that the first simplex of a shelling is always of longest length. If we have an edge-labelling, a maximal chain gives a word by concatenating the labels of the edges we see by following the chain from bottom to top. In this context we call a maximal chain \emph{increasing} if the associated word is increasing. 
An \emph{$EL$-labelling} of a lattice $L$ is an edge-labelling such that in any interval, when reading the labels following the chains from bottom to top, there is a unique maximal increasing chain and the label word of the increasing chain lexicographically precedes the label word of any other maximal chains. Denote $\overline{L}:=L\setminus \{\hat{0},\hat{1}\}$. If $L$ admits an $EL$-labelling, then the order complex $\Delta(\overline{L})$ is shellable and homotopy equivalent to a wedge of spheres.
Any left modular lattice admits an $EL$-labelling (see Remark \ref{rmkEL}). For any left modular chain, there exists a shelling of $\Delta(L)$ that starts with this chain, thus the left modular chains are always of longest length. It follows from the fact that all semidistributive extremal lattice are trim, thus left modular, that $(i)\Longrightarrow (ii) \Longrightarrow (iii) \Longrightarrow (iv)$ of Theorem \ref{thm:mainthmintro} is true for semidistributive lattices. By Lemma \ref{lem:congisSD}, these implications are true in particular for congruence uniform lattices.

\subsection{Order dimension}
\label{sec:orderdimension}

We now recall what is the (order) dimension of a poset $P$. The book \cite{trotter2002combinatorics} of W.T. Trotter is a good reference.

\begin{definition}[\cite{DushnikMiller}]
Let $(\mathbb{R}^n,\leq)$ be the poset defined by $(x_1,\dots,x_n)\leq (y_1,\dots,y_n)$ if $x_i\leq y_i$ for every $i\in [n]$. The \emph{dimension} of a poset $P$ is the smallest integer $n$ such that $P$ is isomorphic to a subposet of $(\mathbb{R}^n,\leq)$.
\end{definition}

Equivalently, this is the minimum number of linear extensions whose intersection is $P$, where the intersection is the poset on the set $P$ whose relations are $x\leq y$ if in each of the chosen linear extensions we have $x\leq y$.
The posets of dimension smaller than $3$ include, for example, the planar posets with a minimum \cite{TROTTERplanarTrees}, the interval orders \cite{RABINOVITCH197850} and the posets whose Hasse diagrams contain at most one cycle \cite{abram2025dimension}. It seems that for families of lattices the dimension is easier to determine. For example, a lattice is planar if and only if its dimension is smaller than $2$ \cite{BakerDim2}. This is false for posets since there exist planar posets of any dimension \cite{KELLY1981135}. 
For some families of semidistributive extremal lattices, we will be able to find their dimension in Section \ref{sec:applicationdimension}.

We finish by recalling a general bound on the dimension of a poset $P$. The \emph{width} of a poset $P$, denoted by $\mathrm{width}(P)$, is the biggest size of an antichain of $P$.
Denote by $\mathrm{Dis}(P)$ the subposet of $P$ on its dissectors, where a \emph{dissector} is an element $x$ such that $P\setminus [x,\hat{1}]$ has a maximum element.

\begin{proposition}[{\cite[Theorem 6]{reading2002order}}]
	\label{prop:bounddimreading}
	For any poset $P$, we have 
	$$\mathrm{width}(\mathrm{Dis}(P)) \leq \dim(P) \leq \mathrm{width}(\mathrm{JIrr}(P)) .$$
\end{proposition}

N. Reading gives two proofs of his Proposition \ref{prop:bounddimreading}. One of these proofs use results that we will recall in Section \ref{sec:dimensionresult}.
This implies the following result, which is proved similarly in \cite[Lemma 4.22]{mcconville2024bubble}.

\begin{lemma}
	\label{prop:dimSDbigger}
	Let $L$ be a semidistributive lattice such that the maximum number of elements that cover an element of $L$ is $n$. Then $\dim(L)\geq n$.   
\end{lemma}

\begin{proof}
	Let $x\in L$ be an element covered by $n$ elements. Let $L'=[x,\hat{1}]$. Then $L'$ is a semidistributive lattice. Any atom of $L'$ is a dissector. Indeed, let $y$ be an atom of $L'$. We have $F_{L'}(x)\setminus F_{L'}(y) = L'\setminus [y,\hat{1}]$. Since $L'$ is semidistributive and $x\lessdot y$ is a cover relation, it follows from Lemma \ref{lemEquiJSD} that $F_{L'}(x)\setminus F_{L'}(y)$ has a maximum element. Thus $L'\setminus [y,\hat{1}]$ has a maximum, which means that $y$ is a dissector. Since each of the $n$ atoms of $L'$ is a dissector, then $\dim(L)\geq \dim(L') \geq \mathrm{width}(\mathrm{Dis}(L'))\geq n$, which finishes the proof.
\end{proof}

\subsection{Representation theory of algebras}
\label{sec:representationtheoryalgebra}

We finish this background by recalling some vocabulary of the representation theory of associative algebras, which will be only used in Section \ref{sec:torsionclassesgentletree}. A good reference is \cite{Assem_Skowronski_Simson_2006}, and for gentle algebras we recommend \cite{brustle2020combinatorics,palu2021non}.

Let $A$ be a finite-dimensional $\mathbb{K}$-algebra, where $\mathbb{K}$ is an algebraically closed field. We denote by $\mathrm{mod} A$ the category of finite-dimensional right $A$-modules. Subcategories are assumed to be full and closed under finite direct sums and summands. The \emph{torsion classes} are the subcategories of $\mathrm{mod} A$ closed by extensions and quotients. They form a complete lattice with meet given by intersection, denoted by $\mathrm{Tors}(A)$, that was extensively studied \cite{demonet2023lattice,thomas2021introduction}. 

A module is called \emph{indecomposable} if it is not isomorphic to a direct sum of two non-zero modules. A \emph{brick} $B$ of $\mathrm{mod} A$ is a $A$-module such that $\mathrm{End}(B)\cong \mathbb{K}$. Two bricks $B$ and $B'$ are \emph{hom-orthogonal} if $\mathrm{Hom}(B,B') = 0$ and $\mathrm{Hom}(B',B) = 0$. A \emph{semibrick} is a set $\{B_i\}_{i\in I}$ of pairwise hom-orthogonal bricks.  The algebra $A$ is \emph{representation finite} if it has a finite number of non-isomorphic indecomposable modules. A \emph{cycle} in the category $\mathrm{mod}\, A$ is a sequence of indecomposable modules $M_0,M_1,\dots,M_r=M_0$ and non-isomorphisms $f_i:M_i\rightarrow M_{i+1}$ for $0\leq i\leq r-1$ where $r>0$. 
It is known that if $A$ is representation finite and has no cycles, then its indecomposable modules are the bricks \cite{draexler1991}. Moreover, in this case $\mathrm{Tors}(A)$ is finite \cite{demonet2019tilting}.

A \emph{quiver} $Q=(Q_0,Q_1,s,t)$ is a directed graph whose vertices are $Q_0$ and its edges, called arrows, are $Q_1$. Moreover, $s$ and $t$ are two maps from $Q_1$ to $Q_0$ that respectively specify what is the source and target of an arrow. We will assume that $Q$ is finite and connected. A \emph{path} of length $m\geq 1$ is a finite sequence of arrows $\gamma_{1,\psi} \gamma_{2} \dots \gamma_m$ with $t(\gamma_{i})=s(\gamma_{i+1})$ for all $i\in [m-1]$. We denote by $\mathbb{K}Q$ the path algebra of $Q$ and by $R_Q$ its ideal generated by arrows. A zero relation in $Q$ is given by a path $\gamma_{1,\psi} \gamma_{2} \dots \gamma_m$ with $m\geq 2$. A two sided ideal $I$ of $\mathbb{K}Q$ is an \emph{admissible ideal} if there exists $m\geq 2$ such that $R_Q^m \subseteq I \subseteq R_Q^2$. In the sequel $I$ is always assumed to be an admissible ideal generated by zero relations. For background on gentle algebras see \cite{brustle2020combinatorics,palu2021non}.

\begin{definition}
An algebra $A=\mathbb{K}Q/I$ is a \emph{gentle algebra} if all the following properties are satisfied:
\begin{enumerate}
    \item[$\bullet$] The ideal $I$ is generated by zero relations of length two.
    \item[$\bullet$] There are at most two incoming and two outgoing arrows at every vertex of $Q$, and there is at most one arrow $\beta_{\psi}$ and one arrow $\gamma$ such that $0\neq \alpha \beta \in I$ and $0\neq \gamma \alpha \in I$.
    \item[$\bullet$] For every arrow $\alpha$, there is at most one arrow $\beta_{\psi}$ and one arrow $\gamma$ such that $\alpha \beta \not\in I$ and $\gamma \alpha \not\in I$.
\end{enumerate}
\end{definition}

The only examples of gentle algebras that we will see are the \emph{gentle trees}, which are the gentle algebras whose underlying undirected graph is a tree (see Figure \ref{fig:gentletreeexample}). 
The gentle trees are representation finite and their indecomposables are certain modules called string modules (see Section \ref{sec:torsionclassesgentletree}). 

A \emph{representation} $M$ of $Q$ is defined by the following data: to each vertex $x\in Q_0$ we associate a $\mathbb{K}$-vector space $M_x$, and to each arrow $\alpha:x\rightarrow y$ in $Q_1$ we associate a $\mathbb{K}$-linear map $\varphi_{\alpha}:M_x \rightarrow M_y$. Such a representation is denoted by $(M_{x},\varphi_{\alpha})$.
If $Q$ has zero relations, we ask additionally that the composition of maps corresponding to a zero relation gives the zero linear map. If $(M_{x},\varphi_{\alpha})$ and $(M'_{x},\varphi'_{\alpha})$ are two representations of $Q$, a morphism $f$ between these two representations is a family $f=(f_x)_{x \in Q_0}$ of $\mathbb{K}$-linear maps that satisfy $\varphi'_{\alpha} f_x = f_y \varphi_{\alpha}$ for all $\alpha : x\rightarrow y$ in $Q_1$.
The representations of $Q$ with their morphisms form an abelian category which is equivalent to the category $\mathrm{mod}\, \mathbb{K}Q$. Thus we will allow ourselves to switch between the two languages.

\section{The main results}
\label{sec:mainresults}

\subsection{Edge-labellings and left modularity} \label{sec:EdgelabellingLeftModular}

In this section we give a result that characterizes left modularity as the equality of some edge-labellings. 

\begin{definition}
	\label{deflabelling}
	Let $L$ be a lattice. Denote by $\psi :\, \hat{0}=x_0 <\dots <x_k=\hat{1}$ a chain containing $\hat{0}$ and $\hat{1}$. Using $\psi$, we have four labellings $\gamma_{1,\psi}$, $\gamma_{1',\psi}$, $\gamma_{2,\psi}$ and $\gamma_{2',\psi}$ of the cover relations of $L$ (the edges of its Hasse diagram). For $j\in \mathrm{JIrr}(L)$, denote $\delta_{\psi}(j):=\mathrm{min}\{i\,|\,j\leq x_i\}$. For $m\in \mathrm{MIrr}(L)$, denote $\beta_{\psi}(m):=\mathrm{max}\{i\,|\,m\geq x_{i-1}\}$. Then, for a cover relation $b\lessdot c$ we define
	\begin{align*}
		\gamma_{1,\psi} (b\lessdot c) &:= \mathrm{min}\{ \delta_{\psi}(j) \,|\, \text{$j$ join-irreducible},\,j\leq c,\,j\not\leq b\}, \\
		\gamma_{1',\psi} (b\lessdot c) &:= \mathrm{max}\{i \,|\, c \wedge x_{i-1} \leq b\} ,\\
		\gamma_{2,\psi}  (b\lessdot c) &:= \mathrm{max}\{ \beta_{\psi}(m) \,|\, \text{$m$ meet-irreducible},\,m\geq b,\,m\not\geq c\} ,\\
		\gamma_{2',\psi} (b\lessdot c) &:= \mathrm{min}\{i \,|\, b \vee x_i \geq c\} .
	\end{align*}
\end{definition}

\begin{remark} \label{rmk:gammaSDgammaLM}
	Recall the edge-labellings $\gamma_J$ and $\gamma_M$ from Section \ref{sec:Backgroundext}. If $L$ is semidistributive, then $\gamma_J(b\lessdot c)=j \Longrightarrow  \gamma_{1,\psi}(b\lessdot c)=\delta_{\psi}(j)$ and $\gamma_M(b\lessdot c)= m \Longrightarrow \gamma_{2,\psi}(b\lessdot c) = \beta_{\psi}(m)$.
\end{remark}

The following was partly stated as Theorem \ref{thm:labellingsintro} in the introduction:

\begin{theorem}
	\label{thmlabelling}
	Using the notations from Definition \ref{deflabelling}, for any lattice $L$, we have $\gamma_{2,\psi}  = \gamma_{2',\psi} \leq \gamma_{1,\psi} = \gamma_{1',\psi}$. Moreover $\gamma_{1',\psi} = \gamma_{2',\psi}$ if and only if for all $i$, $x_i$ is left modular.
\end{theorem}

Before giving the proof of Theorem \ref{thmlabelling}, we give two lemmas.

\begin{lemma}
	\label{lemmaEgalityGamma}
	Using the same notations from Definition \ref{deflabelling}, we have $\gamma_{1,\psi} = \gamma_{1',\psi}$ and $\gamma_{2,\psi} = \gamma_{2',\psi}$.
\end{lemma}

\begin{figure}
	\begin{subfigure}{.5\textwidth}
		\centering
		\begin{tikzpicture}
			\begin{scope}[scale= 0.8]
				\draw (0,0)--(1,2)--(0,4)--(-1,2.8)--(-1,1.2)--(0,0);
				\draw (0,0) node[below]{$x_0$};
				\draw (0,0) node[red]{$\bullet$};
				\draw (0,4) node[above]{$x_3$};
				\draw (0,4) node[red]{$\bullet$};
				\draw (-1,2.8) node[left]{$x_2$};
				\draw (-1,2.8) node[red]{$\bullet$};
				\draw (-1,1.2) node[left]{$x_1$};
				\draw (-1,1.2) node[red]{$\bullet$};
				
				\draw (0.5,1) node[below right]{$3$};
				\draw (0.5,1) node[above left]{$3$};
				\draw (0.5,3) node[above right]{$1$};
				\draw (0.5,3) node[below left]{$1$};
				\draw (-0.6,3.4) node[above left]{$3$};
				\draw (-0.6,3.4) node[below right]{$3$};
				\draw (-1,2) node[left]{$2$};
				\draw (-1,2) node[right]{$2$};
				\draw (-0.5,0.6) node[below left]{$1$};
				\draw (-0.5,0.6) node[above right]{$1$};
			\end{scope}
			
			\begin{scope}[xshift = 4cm, scale=0.9]
				\draw (0,0)--(2,1)--(0,5)--(-2,1)--(0,0);
				\draw (0,0)--(0,1)--(1,3);
				\draw (0,1)--(-1,3);
				
				\draw (0,0) node[red]{$\bullet$};
				\draw (2,1) node[red]{$\bullet$};
				\draw (1,3) node[red]{$\bullet$};
				\draw (0,5) node[red]{$\bullet$};
				
				\draw (0,0) node[below]{$x_0$};
				\draw (2,1) node[right]{$x_1$};
				\draw (1,3) node[above right]{$x_2$};
				\draw (0,5) node[above]{$x_3$};
				
				\draw[blue, ultra thick] (-2,1)--(-1,3);
				\draw (-2,1) node[left,blue]{$b$};
				\draw (-1,3) node[above left,blue]{$c$};
				
				\draw (1,0.5) node[below right]{$1$};
				\draw (1,0.5) node[above left]{$1$};
				\draw (0,0.5) node[right]{$2$};
				\draw (0,0.5) node[left]{$2$};
				\draw (1.5,2) node[above right]{$2$};
				\draw (1.5,2) node[below left]{$2$};
				\draw (-1.5,2) node[above left, blue]{$2$};
				\draw (-1.5,2) node[below right, blue]{$1$};
				\draw (0.5,2) node[above left]{$1$};
				\draw (0.5,2) node[below right]{$1$};
				\draw (-0.5,4) node[above left]{$1$};
				\draw (-0.5,4) node[below right]{$1$};
				\draw (0.5,4) node[above right]{$3$};
				\draw (0.5,4) node[below left]{$3$};
				\draw (-0.5,2) node[above right]{$3$};
				\draw (-0.5,2) node[below left]{$3$};
				\draw (-1,0.5) node[above right]{$3$};
				\draw (-1,0.5) node[below left]{$3$};
			\end{scope}
		\end{tikzpicture}
		
		\caption{$\gamma_{1',\psi}$ at the left of an edge, $\gamma_{2',\psi}$ at the right.\\
			On the right $x_1$ is not left modular \\
			because of $b\lessdot c$.}
		\label{fig:sfig1}
	\end{subfigure}%
	\begin{subfigure}{.5\textwidth}
		\centering
		\begin{tikzpicture}
			\begin{scope}[scale=0.7]
				\draw (0,0)--(2,2)--(2,4)--(0,6)--(-2,4)--(-2,2)--(0,0);
				\draw (-2,2)--(2,4);
				\draw (2,2)--(-2,4);
				\draw[blue, ultra thick] (0,0)--(-2,2)--(-2,4)--(0,6);
				\draw (-2.5,3) node[left,blue]{$\phi$};
				\draw[red] (2,3) ellipse (0.6cm and 1.7cm);
				\draw (3,2) node[red]{$C$};
				\draw (2,2) node[red]{$\bullet$};
				\draw (2,4) node[red]{$\bullet$};
			\end{scope}
		\end{tikzpicture}
		\caption{The elements of the blue maximal chain $\phi$ are all comparable to at least one element in $C$.}
		\label{fig:sfig2phi}
	\end{subfigure}
	\caption{}
	\label{fig:fig}
\end{figure}

\begin{proof}
	Let $b\lessdot c$ be a cover relation.
	We only prove $\gamma_{2,\psi}  = \gamma_{2',\psi}$ as the argument for $\gamma_{1,\psi}  = \gamma_{1',\psi}$ is dual.
	
	Denote $i_0:=\gamma_{2',\psi} (b\lessdot c) =\mathrm{min}\{i \,|\, b \vee x_i \geq c\}$. It means that $b\vee x_{i_0} \geq c$ but $b\vee x_{i_0-1} \not\geq c$. Let us prove that $i_0=\gamma_{2,\psi}(b\lessdot c)$.
	
	First, let us prove that $\gamma_{2,\psi}(b\lessdot c)\leq i_0$. We just have to prove that for any $m$ meet-irreducible such that $m\geq b$ and $m\not\geq c$, we have $\beta_{\psi}(m)\leq i_0$. Seeking a contradiction, suppose $\beta_{\psi}(m)> i_0$. Then by definition of $\beta_{\psi}$ we have $m\geq x_{i_0}$ since the $x_i$'s form a chain. Since $m\geq b$ and $m\geq x_{i_0}$, then $m\geq b\vee x_{i_0} \geq c$. It is absurd because $m\not\geq c$, which proves that $\gamma_{2,\psi}(b\lessdot c)\leq i_0$.
	
	Now, let us prove that $i_0 \leq \gamma_{2,\psi}(b\lessdot c)$. By definition of $\gamma_{2,\psi}$, we want to find a meet-irreducible $m$ such that $m\geq b$, $m\not\geq c$ and such that $i_0\leq \beta_{\psi}(m)$. The latter condition is equivalent to $m\geq x_{i_0-1}$. So we want to find a meet-irreducible $m$ such that $m\geq b\vee x_{i_0-1}$ and $m\not\geq c$. Such a meet-irreducible exists because $b\vee x_{i_0-1} \not\geq c$, as it was noted in the beginning. If such a meet-irreducible does not exist, then all the meet-irreducibles above $b\vee x_{i_0-1}$ would also be above $c$, and since an element is the meet of the meet-irreducibles above it, then $b\vee x_{i_0-1} \geq c$, which is absurd. Thus $i_0 \leq \gamma_{2,\psi}(b\lessdot c)$. This finishes the proof of $\gamma_{2,\psi}=\gamma_{2',\psi}$. 
\end{proof}

\begin{lemma}
	\label{lemlabelling}
	Using the same notation as in Definition \ref{deflabelling}, we have $\gamma_{2',\psi} \leq \gamma_{1',\psi}$.
\end{lemma}

\begin{proof}
	Let $b\lessdot c$ be a cover relation. Denote $i:=\gamma_{2',\psi}(b\lessdot c)$. It means that $b\vee x_i \geq c$ and $b\vee x_{i-1} \not\geq c$. We want to prove that $\gamma_{1',\psi}(b\lessdot c) \geq i$, which is equivalent to $c\wedge x_{i-1} \leq b$. Seeking a contradiction, suppose that $c\wedge x_{i-1} \not\leq b$.
	Since $b<c$ and $c\wedge x_{i-1} \leq c$, we have $b\leq b\vee (c\wedge x_{i-1}) \leq c$. Since $c\wedge x_{i-1} \not\leq b$, then $b\vee (c\wedge x_{i-1}) \neq  b$. Thus $b \vee (c\wedge x_{i-1}) = c$. But $b\vee x_{i-1} \geq b\vee (c\wedge x_{i-1}) =c$, which is absurd. 
\end{proof}

\begin{proof}[Proof of Theorem \ref{thmlabelling}]
	Lemma \ref{lemmaEgalityGamma} and Lemma \ref{lemlabelling} give the first part of Theorem \ref{thmlabelling}. For the rest, using Lemma \ref{lemlabelling} we have    
	\begin{align*}
		\gamma_{1',\psi} \neq \gamma_{2',\psi} \iff \exists i,\,\exists \, b\lessdot c, \,\,
		\left\{\begin{array}{cc}
			\gamma_{2',\psi}(b\lessdot c)\leq i\\
			\gamma_{1',\psi}(b\lessdot c)\geq i+1 .
		\end{array}
		\right.  
	\end{align*}
	
	By definition of $\gamma_{1',\psi}$ and $\gamma_{2',\psi}$, it is also equivalent to the fact that there exist an $i$ and a cover $b\lessdot c$ such that the following subposet of $L$ is a sublattice:
	
	\centering
	\begin{tikzpicture}
		\begin{scope}[scale=0.6]
			\draw[dashed] (0,0)--(1,2)--(0,4)--(-1,2.8)--(-1,1.2)--(0,0);
			\draw (-1,2.8)--(-1,1.2);
			\draw (0,0) node[below]{$c\wedge x_i$};
			\draw (1,2) node[right]{$x_i$};
			\draw (0,4) node[above]{$b\vee x_i$};
			\draw (-1,2.8) node[left]{$c$};
			\draw (-1,1.2) node[left]{$b$};
		\end{scope}    
	\end{tikzpicture}
	
	Using $(4)$ of Proposition \ref{Propleftmodular}, we deduce that $\gamma_{1',\psi}\neq \gamma_{2',\psi}$ is equivalent to the fact that there exists $i$ such that $x_i$ is not left modular. This finishes the proof.
\end{proof}

\begin{remark} \label{rmkEL}
	S.-C. Liu proved in \cite{LiuLeftmodular}, for the case of a longest chain $\psi$, that $\gamma_{2',\psi}\leq \gamma_{1,\psi}$ and that if for all $i$, $x_i$ is left modular, then $\gamma_{2',\psi}=\gamma_{1,\psi}$. He also proved that the equal labellings that we obtain starting from a left modular chain are $EL$-labellings. What we add to the story, other than an easier proof (by considering $\gamma_{1',\psi}$ we also add some symmetry), is a way to use these labellings to prove that lattices are left modular.
	
	P. McNamara and H. Thomas \cite{MCNAMARA2006101} proved that the above equal labellings of a left modular lattice are interpolating labellings (an $EL$-labelling with an additional property), and also proved that if a finite lattice has an interpolating labelling, then it is left modular. Our Theorem \ref{thmlabelling} is then of the same flavor, helping us prove left modularity, but without mentioning $EL$-labelling.
\end{remark}

Recall the definition of the bijection $\kappa$ from Section \ref{sec:Backgroundext}. We can use Theorem \ref{thmlabelling} to give a short proof of the following result:

\begin{corollary}[{\cite[Theorem 1.4]{Thomas_2019}}]
	\label{CorExtremalSD}
	Semidistributive extremal lattices are left modular.
\end{corollary}

\begin{proof}
	Since $L$ is extremal, the choice of any longest chain $\psi: \,\hat{0}=x_0 \lessdot \dots \lessdot x_n=\hat{1}$ gives a numbering of the join and meet-irreducibles $j_1,j_2,\dots,j_n$ and $m_1,m_2,\dots,m_n$ such that 
	$x_i=j_1\vee \dots \vee j_i = m_{i+1}\wedge \dots \wedge m_n$. We have $\kappa(j_i)=m_i$ for all $i\in [n]$ by Proposition \ref{prop:kappaextremalSD}. Moreover $\delta_{\psi}(j_i)=i$ and $\beta_{\psi}(m_i)=i$.
	Since $L$ is semidistributive, using Proposition \ref{prop:kappaSD} we have that $\gamma_J(b\lessdot c)=j_i$ if and only if $\gamma_M(b\lessdot c)=m_i$, for any $b\lessdot c$. Thus with Remark \ref{rmk:gammaSDgammaLM}, we have $\gamma_{1,\psi}=\gamma_{2,\psi}$. Then Theorem \ref{thmlabelling} gives that $L$ is left modular.
\end{proof}

\begin{remark}
	Theorem \ref{thmlabelling} was proved and first used in order to obtain the left modularity of the $(P,\phi)$-Tamari lattices (defined by the author in \cite{segovia2025pphitamarilattices}).
\end{remark}

\subsection{Extremality and left modularity of congruence normal lattices} \label{sec:congNormal}

In this section, we characterize the congruence normal lattices that are extremal or left modular by necessary and sufficient conditions on each doubling step.
In the sequel, unless stated otherwise, $C$ is a convex subset of a lattice $L$.

\begin{definition}
	We call the $\emph{heart}$ of $C$, written $H(C)$, all the elements of $C$ that are less than all the maximal elements of $C$ and bigger than all its minimal elements.
    \end{definition}

\begin{lemma}
	\label{lemHeart}
	For any convex subset $C$ of $L$,
	\begin{itemize}
		\item[$(i)$] $H(C)\neq \emptyset$ if and only if the meet of the maximal elements is bigger than the join of the minimal elements of $C$. In this case, if these elements are respectively given by $M_i$ for $i\in I$ and $m_j$ for $j\in J$, we have $H(C)=[\vee_{i\in I} m_i \,,\, \wedge_{j\in J} M_j]$.
		\item[$(ii)$] $H(C)$ is convex.
		\item[$(iii)$] $H(C)=C$ if and only if $C$ is an interval.
		\item[$(iv)$] $H(C)=\{a\in C\,|\,\forall b\in C,\,a\vee b\in C \,\,\text{and}\,\,a\wedge b\in C \}$.
	\end{itemize}
\end{lemma}

\begin{proof}
	The statement $(i)$ is easy to prove and implies $(ii)$ and $(iii)$. Let us prove $(iv)$.
	
	Let $a\in H(C)$. Let $b\in C$. There exists a minimal element $m\in C$ and a maximal element $M\in C$ such that $m\leq b\leq M$. Since $a\in H(C)$, we have $m\leq a\leq M$. Then $a\leq a\vee b \leq M$, thus by convexity of $C$ we have $a\vee b\in C$. Also $m \leq a\wedge b \leq a$ implies by convexity that $a\wedge b\in C$. Thus $a\in \{a\in C\,|\,\forall b\in C,\,a\vee b\in C \,\,\text{and}\,\,a\wedge b\in C \}$.
	
	Now let $a\in C$ such that for all $b\in C$, we have $a\vee b\in C$ and $a\wedge b\in C$. Let $M$ be a maximal element and $m$ a minimal element of $C$. We have by hypothesis that $a\vee M\in C$ and $a\wedge m \in C$. But since $M$ is maximal and $m$ is minimal in $C$, this happens if and only if $a\leq M$ and $a\geq m$. Then $a$ is smaller than all the maximal elements of $C$ and bigger than all the minimal elements of $C$. Thus $a\in H(C)$.
\end{proof}

The next two results were partly stated as Theorem \ref{thm:extremalleftmodularintro} in the introduction:

\begin{proposition}
	\label{PropExtremality}
	Let $L=E[C_1,\dots,C_n]$ be a congruence normal lattice. Then
	\begin{itemize}
		\item[$(1)$] $L$ is join-extremal if and only if it is join-congruence uniform and $C_i$ intersects the spine of $E[C_1,\dots,C_{i-1}]$, for all $i$.
		\item[$(2)$] We have the dual result of $(1)$ for meet-extremality by replacing \emph{join} by \emph{meet}.
		\item[$(3)$] $L$ is extremal if and only if it is congruence uniform and $C_i$ intersects the spine of $E[C_1,\dots,C_{i-1}]$, for all $i$.
	\end{itemize}
\end{proposition}

\begin{proof}
	Since $L$ is congruence normal we construct it by successive doublings of convex subsets starting from the one element lattice, which is extremal. By Lemma \ref{lemJIRR}, at each doubling of the convex subset $C_i$, we add exactly as many join-irreducibles as the number of minimal elements of $C_i$, and we add as many meet-irreducibles as the number of maximal elements of $C_i$. By Lemma \ref{lem:lengthdoublement} the doubling of $C_i$ will add $0$ or $1$ to the length of $L$; $0$ if $C_i$ does not contain an element of the spine and $1$ otherwise. Thus, we keep the initial join-extremality through the doubling procedure if and only if at each step $C_i$ is a lower-pseudo interval that contains an element of the spine. Indeed, these are the subsets that have a minimum element. Similarly for meet-extremality with upper pseudo-intervals. Finally, for extremality we need $C_i$ to have a minimum and a maximum, and since it is convex this means that it has to be an interval. This finishes the proof.
\end{proof}

\begin{theorem}
	\label{thmCongruencenormal}
	Let $L=E[C_1,\dots,C_n]$ be a congruence normal lattice. Then $L$ is left modular if and only if $H(C_i)$ has an element that lies on a left modular chain of $E[C_1,\dots,C_{i-1}]$, for all $i$. 
\end{theorem}

To prove Theorem \ref{thmCongruencenormal}, we will need some lemmas. The next two results follow from Definition \ref{defDoublementDay}.

\begin{lemma}
	\label{lemElmmodDoublement}
	Let $c\in C$ and $(a,\epsilon)\in L[C]$. We have 
	\begin{align*}
		(a,\epsilon)\vee (c,0) \neq (a,\epsilon)\vee (c,1) &\iff \left(a\vee c \in C\,\,\text{and}\,\,\epsilon=0\right)\\
		(a,\epsilon)\wedge (c,0) \neq (a,\epsilon)\wedge (c,1) &\iff \left(a\wedge c \in C\,\,\text{and}\,\,\epsilon=1\right)
	\end{align*}
\end{lemma}

\begin{lemma} \label{lem:coverdoubllem}
	In $L[C]$ we have a cover $(b,0)\lessdot (c,1)$ if and only if $b=c\in C$, or $b\in I_L(C) \setminus C$, $c\not\in I_L(C)$ and $b\lessdot c$ is a cover in $L$.    
\end{lemma}

\begin{lemma}
	\label{lemModcompareLtoLC}
	The element $(a,\epsilon)\in L[C]$ does not have a cover $(b,\alpha)\lessdot (c,\alpha')$ with $b\neq c$ that is a counterexample to its left modularity if and only if $a\in L$ is left modular.
\end{lemma}

\begin{proof}
	Suppose that $a\in L$ is not left modular because of $b\lessdot c$. By $(3)$ of Proposition \ref{Propleftmodular} this means $a\vee b = a\vee c$ and $a\wedge b=a\wedge c$. We use the injection $\psi: L\rightarrow L[C]$ following Definition \ref{defDoublementDay} and define $\alpha$ and $\alpha'$ by $(b,\alpha)=\psi(b)$ and $(c,\alpha')=\psi(c)$, unless $b\in C$ and $c\not \in C$, in which case we define $\alpha=\alpha'=1$. These choices guarantee that $(b,\alpha)\lessdot (c,\alpha')$ is a cover of $L[C]$. We have:
	
	\begin{align}
		(a,\epsilon)\vee (b,\alpha) \neq (a,\epsilon)\vee (c,\alpha') &\iff (a\vee c\in I_L(C),\,\epsilon=0,\,\alpha=0\,\,\text{and}\,\,\alpha'=1) \label{eq:mod1}   \\
		(a,\epsilon)\wedge (b,\alpha) \neq (a,\epsilon)\wedge (c,\alpha') &\iff (a\wedge b\in \big(L\setminus I_L(C)\big)\cup C,\,\epsilon=1,\,\alpha=0\,\,\text{and}\,\,\alpha'=1)  \label{eq:mod2}
	\end{align}
	
	We want to conclude that $(a,\epsilon)\vee (b,\alpha) = (a,\epsilon)\vee (c,\alpha')$ and $(a,\epsilon)\vee (b,\alpha) = (a,\epsilon)\vee (c,\alpha')$, so it suffices to show that one of the conditions on the RHS of \eqref{eq:mod1} and \eqref{eq:mod2} is false.
	Suppose that $\alpha=0$ and $\alpha'=1$. Then by Lemma \ref{lem:coverdoubllem}, since $b\neq c$ we have $b\in I_L(C)\setminus C$ and $c\not\in I_L(C)$. Thus $a\vee c\not\in I_L(C)$ and $a\wedge b \not\in \big(L\setminus I_L(C)\big)\cup C$, which concludes.
	
	It follows from $(3)$ of Proposition \ref{Propleftmodular} that $(b,\alpha)\lessdot (c,\alpha')$ with $b\neq c$ is a counterexample to the left modularity of $(a,\epsilon)$ in $L[C]$.
	
	Conversely, assume that $a\in L$ is left modular. Let $(b,\alpha)\lessdot (c,\alpha')$ be a cover in $L[C]$ such that $b\neq c$. Since $b\neq c$, we have by Lemma \ref{lem:coverdoubllem} that $b\lessdot c$ is a cover in $L$. Since $a$ is left modular it implies by $(3)$ of Proposition \ref{Propleftmodular} that $a\vee b\neq a\vee c$ or $a\wedge b\neq a\wedge c$. By definition of the join in $L[C]$, it is clear that $(a,\epsilon)\vee (b,\alpha) \neq (a,\epsilon)\vee (c,\alpha')$ or $(a,\epsilon)\wedge (b,\alpha) \neq (a,\epsilon)\wedge (c,\alpha')$, which again using $(3)$ of Proposition \ref{Propleftmodular} finishes the proof.
\end{proof}

\begin{figure}
	\centering
	\begin{tikzpicture}
		\begin{scope}[scale=0.9]
			\draw (0,0)--(1,1)--(0,2)--(-1,1)--(0,0);
			\draw[ultra thick, red] (-1,1)--(0,0)--(1,1);
			\draw (0,0) node[blue]{$\bullet$};
			\draw (0,2) node[blue]{$\bullet$};
			\draw (1,1) node[blue]{$\bullet$};
			\draw (-1,1) node[blue]{$\bullet$};
			\draw[thick] (0,0) circle (0.3);
			
			\begin{scope}[xshift=3cm, yshift= 1.5cm]
				\draw (0,0)--(1,1)--(0,2)--(-1,1)--(0,0);  
				\draw (-1,1)--(-1,-0.5)--(0,-1.5)--(1,-0.5)--(1,1);
				\draw[ultra thick, red] (1,-0.5)--(1,1);
				\draw (0,0) node[blue]{$\bullet$};
				\draw (0,2) node[blue]{$\bullet$};
				\draw (1,1) node[blue]{$\bullet$};
				\draw (-1,1) node[blue]{$\bullet$};
				\draw (0,-1.5)--(0,0);
				\draw (0,-1.5) node[blue]{$\bullet$};
				\draw[thick] (1,0.25) ellipse (0.4cm and 1.1cm);
			\end{scope}
			
			\begin{scope}[xshift=6.5cm, yshift= 1.5cm]
				\draw (-1,1)--(0,0)--(1,1);  
				\draw (0,0)--(0,-1.5);
				\draw (-1,1)--(-1,-0.5)--(0,-1.5)--(1,-0.5)--(1,1)--(2,2)--(2,0.5)--(1,-0.5);   
				\draw (-1,1)--(1,3)--(2,2);
				\draw[ultra thick,red] (1,1)--(1,-0.5)--(2,0.5);
				\draw (0,0) node[blue]{$\bullet$};
				\draw (0,-1.5) node[blue]{$\bullet$};
				\draw (1,1) node[blue]{$\bullet$};
				\draw (2,2) node[blue]{$\bullet$};
				\draw (1,3) node[blue]{$\bullet$};
				\draw[thick] (1,-0.5) circle (0.3);
			\end{scope}
			
			\begin{scope}[xshift=10.5cm, yshift= 1.5cm]
				\draw (-1,1)--(0,0)--(1,1); 
				\draw (0,0)--(0,-1.5);
				\draw (-1,1)--(-1,-0.5)--(0,-1.5)--(1,-0.5);
				\draw (1,1)--(1,-0.5)--(1.8,0.3);
				\draw (1.5,2.2)--(1.5,0.7)--(2.3,1.5);
				\draw (1,1)--(1.5,2.2);  
				\draw (1,-0.5)--(1.5,0.7);  
				\draw (1.8,0.3)--(2.3,1.5);  
				\draw (1.5,2.2)--(2.3,3);
				\draw (2.3,1.5)--(2.3,3);
				\draw (-1,1)--(1.3,4)--(2.3,3);  
				\draw (1.5,2.2) node[blue]{$\bullet$};
				\draw (0,-1.5) node[blue]{$\bullet$};
				\draw (2.3,3) node[blue]{$\bullet$};
				\draw (1.3,4) node[blue]{$\bullet$};
			\end{scope}
		\end{scope}
	\end{tikzpicture}
	\caption{We represent $3$ successive doublings. The left modular elements are the blue dots. The thick red edges form the convex subsets $C$ that we double and we circled the elements of $H(C)$.}
	\label{fig:doublingsLeftmodular}
\end{figure}
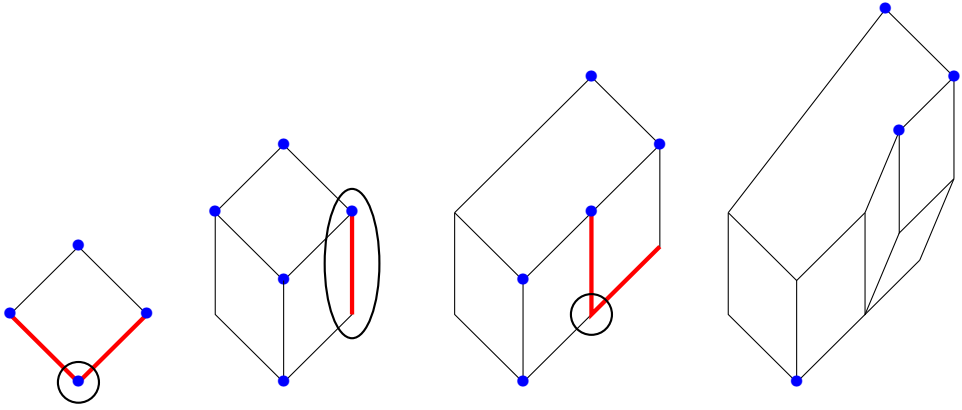

We can now prove a characterization of the left modular elements of the doubling $L[C]$.

\begin{proposition}
	\label{prop:TechniqueMod}
	We have $(a,0)\in L[C]$ is left modular if and only if $a\in L$ is left modular and $a$ is smaller than all the maximal elements of $C$.
	We have $(a,1)\in L[C]$ is left modular if and only if $a\in L$ is left modular and $a$ is bigger than all the minimal elements of $C$.
	Thus $(a,0)\lessdot (a,1)$ with $a\in C$ are two left modular elements if and only if $a\in L$ is left modular and $a\in H(C)$. 
\end{proposition}

\begin{proof}
	Lemma \ref{lemModcompareLtoLC} tells us that the left modularity of $(a,\epsilon) \in L[C]$ is seen, if $a\in L$ is left modular, by looking only at the covers of the doubling $(c,0)\lessdot (c,1)$ with $c\in C$. We described what happens with these covers in Lemma \ref{lemElmmodDoublement}. Thus by $(3)$ of Proposition \ref{Propleftmodular} we obtain that $(a,0)\in L[C]$ is left modular if and only if $a\in L$ is left modular and $a\vee c \in C$ for all $c\in C$. This latter condition is equivalent to the fact that $a$ is smaller than all the maximal elements of $C$. We deduce similarly the case of $(a,1)$. The last part with $(a,0)\lessdot (a,1)$ is a direct consequence of the first part.
\end{proof}

Now that we understand the left modularity of the elements of $L[C]$, we still need to look at the left modular chains. 
The next result is true for any lattice. A particular case gives that any maximal chain that does not intersect an interval has an element incomparable to that interval, which is not true in a general poset (see Figure \ref{fig:sfig2phi}).

\begin{lemma}
	\label{lemChainemaxIncomparable}
	Let $C$ be a convex subset of $L$.
	Any maximal chain of $L$ that does not intersect $C$ has an element that is neither smaller than all the maximal elements of $C$ nor bigger than all its minimal elements. 
\end{lemma}

\begin{proof}
	Let $\Delta$ be a maximal chain of $L$ that does not intersect $C$.
	The elements of this chain are split between $3$ disjoint totally ordered sets : the set $A$ of these that are less than all the maximal elements of $C$, the set $I$ of these that are not less than all the maximal elements nor bigger than all the minimal elements of $C$, and the set $B$ of these that are bigger than all the minimal elements of $C$. By convexity of $C$, we have $A\cap B=\emptyset$. Since $\Delta$ is maximal it contains $\hat{0}$ and $\hat{1}$, thus $A\neq \emptyset$ and $B\neq \emptyset$. We want to prove that $I\neq \emptyset$. Seeking a contradiction, suppose that $I=\emptyset$. Then $\Delta$ is partitioned between $A$, and $B$ which consists of the bigger elements. Since $\Delta$ is maximal we have a cover relation $a \lessdot b$ with $a\in A$ and $b\in B$, giving us the following situation where $m$ and $M$ are minimal and maximal elements of $C$ respectively such that $m\leq M$, $a<M$ and $m<b$:
	
	\begin{center}
		\begin{tikzpicture}
			\begin{scope}[scale=0.6]
				\draw (0,0)--(0,2);
				\draw[dashed] (0,0)--(3,3);
				\draw[dashed] (3,-1)--(3,3);
				\draw[dashed] (3,-1)--(0,2);
				\draw (0,0) node[below]{$a$};
				\draw (0,2) node[above]{$b$};
				\draw (3,-1) node[below]{$m$};
				\draw (3,3) node[above]{$M$};
				\draw (3,1) ellipse (1cm and 2.8cm);
				\draw (4.5,0) node{$C$};
			\end{scope}
		\end{tikzpicture}    
	\end{center}

	If $a< m$, then $a< m < b$, which is absurd since $a\lessdot b$ is a cover relation. Thus $a\not < m$.
	If $m<a$, then by convexity of $C$ we have $m<a<M$ implies $a\in C$, which is absurd. Thus $m\not < a$. Thus, $a$ and $m$ are incomparable. Then $a\vee m = b$. Since $M$ is bigger than $a$ and $m$, it implies $M\geq a\vee m=b$. Thus $m< b< M$, which by convexity of $C$ implies $b\in C$. This is absurd, which proves that $I\neq \emptyset$. 
\end{proof}

\begin{lemma}
	\label{lemChainNOTleftmodular}
	Let $L$ be a lattice and $C$ be a convex subset that does not contain an element of any left modular chain of $L$.  Then $L[C]$ is not left modular.
\end{lemma}

\begin{proof}
	We need to prove that $L[C]$ does not have a left modular chain. 
	
	To begin with, no such chain goes through the doubling, meaning no such chain contains an edge $(c,0)\lessdot (c,1)$ with $c\in C$. Indeed, if it was the case, take the projection of this chain to $L$, meaning we replace the elements $(a,\epsilon)$ of this chain by $a$. By forgetting the repetition of $c$ coming from $(c,0)\lessdot (c,1)$, by Lemma \ref{lemModcompareLtoLC} this would give a left modular chain of $L$. It is absurd because $C$ would contain the element $c$ of this chain but $C$ does not contain any element of a left modular chain of $L$ by hypothesis.
	
	It remains to prove that we do not have a left modular chain of $L[C]$ that does not go through the doubling. By Lemma \ref{lemChainemaxIncomparable} for $C\times \{0\}$, if such a chain exists we would have an element $(a,\epsilon)$ of this chain that is neither smaller than all the maximal elements of $C\times \{0\}$ nor bigger than all its minimal elements. It means that $a$ is neither smaller than all the maximal elements of $C$ nor bigger than all its minimal elements. By Lemma \ref{prop:TechniqueMod} the element $(a,\epsilon)$ would not be left modular in $L[C]$. Then such a left modular chain does not exist.
\end{proof}

Now we know that if $L[C]$ is left modular it is necessary that there exists a left modular chain $\Delta$ in $L[C]$ that goes through the doubling, meaning it comes from a left modular chain $\Delta$ in $L$ that intersects $C$ and by Lemma \ref{prop:TechniqueMod} this intersection contains an element of $H(C)$. The following lemma proves that this is also a sufficient condition.

\begin{lemma}
	\label{chainleftmodular}
	Let $L$ be a lattice and $C$ a convex subset of $L$. Suppose that there exists $a\in H(C)$ that lies on a left modular chain $\Delta:\,x_0\lessdot x_1 \lessdot \dots \lessdot x_l=a\lessdot \dots \lessdot x_n$ of $L$. Then the maximal chain 
	\begin{align}
		\label{chainMod}
		(x_0,0)\lessdot \dots \lessdot (x_{l-1},0) \lessdot (a,0) \lessdot (a,1)\lessdot (x_{l+1},1) \lessdot \dots \lessdot (x_n,1)
	\end{align}
	of $L[C]$ is a left modular chain. Thus $L[C]$ is left modular.
\end{lemma}

\begin{proof}
	Since $a\in L$ is left modular and $a\in H(C)$, we know using Proposition \ref{prop:TechniqueMod} that $(a,0)\lessdot (a,1)$ is a cover relation made of two left modular elements of $L[C]$. Since \eqref{chainMod} is a maximal chain, it just remains to prove that any of its elements other than $(a,0)$ and $(a,1)$ are left modular in $L[C]$. It follows from Lemma \ref{prop:TechniqueMod} since for any other element $(x_i,\epsilon)$ of the chain \eqref{chainMod} we have $x_i\in L$ is left modular and either smaller or bigger than $a\in H(C)$. 
\end{proof}

We can now prove Theorem \ref{thmCongruencenormal}.

\noindent
{\it Proof of Theorem \ref{thmCongruencenormal}.}
	By Lemma \ref{chainleftmodular} we know that if for all $i$ we have that $H(C_i)$ has an element that lies on a left modular chain of $E[C_1,\dots,C_{i-1}]$, then $L$ is left modular. 
	
	Conversely, suppose that $E[C_1,\dots,C_n]$ is left modular. Let $i\leq n$. We want to prove that $H(C_i)$ has an element that lies on a left modular chain of $E[C_1,\dots,C_{i-1}]$. By Lemma \ref{lemChainNOTleftmodular}, it is required that $C_i$ contains an element of a left modular chain of $E[C_1,\dots,C_{i-1}]$. Indeed, if at some point no such chain exists in our lattice then it is too late; we will never recover a left modular lattice by doublings. 
	As recalled at the end of Section \ref{sec:Backgroundext}, all left modular chains are of longest length. With the previous requirement, we know that at each doubling step we add one to the length of the lattice. Then to have a left modular chain after the doubling, since it is a longest chain, we need it to go through the doubling. Thus by Lemma \ref{prop:TechniqueMod} it is forced that $H(C_i)$ contains an element that lies on a left modular chain of $E[C_1,\dots,C_{i-1}]$. \qed

Before giving some corollaries of our results, we give a lemma.

\begin{lemma}
	\label{lemSpineCong}
	Let $L=E[C_1,\dots,C_n]$ be an extremal congruence uniform lattice. Then all the elements of the spine of $L$ are left modular elements.
\end{lemma}

\begin{proof}
	We prove this by induction on the number of doublings. For the one element lattice, the spine is the lattice itself and the only element is left modular. By induction, assume that all the elements of the spine of $E[C_1,\dots,C_{i-1}]$ are left modular. We want to prove that all the elements of the spine of $E[C_1,\dots,C_i]=E[C_1,\dots,C_{i-1}][C_i]$ are left modular.
	We have the equivalence between the fact that all the elements of the spine are left modular and the fact that any longest chain is a left modular chain. Thus the induction hypothesis gives us that the longest chains of $E[C_1,\dots,C_{i-1}]$ are left modular chains.
	By $(3)$ of Proposition \ref{PropExtremality}, we have that $C_i$ intersects the spine of $E[C_1,\dots,C_{i-1}]$. It follows by Lemma \ref{lem:lengthdoublement} that the longest chains of $E[C_1,\dots,C_i]$ go through the doubling; these correspond to the longest chains of $E[C_1,\dots,C_{i-1}]$ intersecting $C_i$. Precisely, these are the chains \eqref{chainMod} of Lemma \ref{chainleftmodular} for a longest chain $\Delta: x_0\lessdot x_1 \lessdot \dots \lessdot x_l:=a\lessdot \dots \lessdot x_r$ of $E[C_1,\dots,C_{i-1}]$, where $a\in C_i$. Since $C_i$ is an interval, by $(iii)$ of Lemma \ref{lemHeart} we have $H(C_i)=C_i$. Thus $a\in H(C_i)$. As $\Delta$ is a left modular chain by hypothesis, by Lemma \ref{chainleftmodular} the longest chains \eqref{chainMod} are also left modular chains, which finishes the proof.
\end{proof}

\begin{corollary}
	\label{coroCongruentUniformMod}
	If $L=E[C_1,\dots,C_n]$ is congruence uniform, then $L$ is left modular if and only if $C_i$ contains an element of the spine of $E[C_1,\dots,C_{i-1}]$, for all $i$.
\end{corollary}

\begin{proof}
	By $(iii)$ of Lemma \ref{lemHeart}, we have $H(C_i)=C_i$, for all $i$. Thus Theorem \ref{thmCongruencenormal} tells us that $L$ is left modular if and only if $C_i$ has an element that lies on a left modular chain of $E[C_1,\dots,C_{i-1}]$, for all $i$. Then using the fact that the left modular chains are of longest length gives one direction of Corollary \ref{coroCongruentUniformMod}, the other direction is obtained using additionally Lemma \ref{lemSpineCong}.
\end{proof}

\begin{corollary} 
	\label{corocongruenceuniformExtremal}
	Join-congruence uniform left modular lattices are join-extremal. Similarly meet-congruence uniform left modular lattices are meet-extremal.
	For congruence uniform lattices, extremality and left modularity are equivalent.
\end{corollary}

\begin{proof}
	The first statements on join and meet-congruence uniform lattices are a direct consequence of Theorem \ref{thmCongruencenormal} with the fact that left modular chains are of longest length. The statement on congruence uniform lattices is the comparison of $(2)$ of Proposition \ref{PropExtremality} and Corollary \ref{coroCongruentUniformMod}.
\end{proof}

\begin{remark}
	Lemma \ref{lemSpineCong} is a special case of the combination of \cite[Lemma 6]{thomas2005analoguedistributivityungradedlattices} and of \cite[Theorem 1.4]{Thomas_2019}, but is proved in our context by a simple induction.
	
	The first statements of Corollary \ref{corocongruenceuniformExtremal} relative to join and meet-congruence uniform lattices are a special case of \cite[Theorem 3.2]{M_hle_2023}, as these lattices are respectively join and meet-semidistributive as recalled in Lemma \ref{lem:congisSD}. Combining this result with \cite[Theorem 1.4]{Thomas_2019} gives, as was observed in \cite{M_hle_2023}, a generalization of the last statement of Corollary \ref{corocongruenceuniformExtremal} to all semidistributive lattices. 
\end{remark}

\subsection{Shellability of congruence normal lattices} \label{sec:shellabilityresult}

Recall the definitions related to shellability from Section \ref{sec:shellability}.
In the sequel $C$ is always a non-empty convex subset of $L$.
We give a necessary condition for a congruence normal lattice to be shellable.

\begin{theorem}  \label{thm:CNextremalshellable}
	Let $L=E[C_1,\dots,C_n]$ be a congruence normal lattice. If $L$ is shellable, then $C_{i}$ intersects the spine of $E[C_1,\dots,C_{i-1}]$, for all $i$.    
\end{theorem}

We will prove Theorem \ref{thm:CNextremalshellable} in the rest of this section, but we start with two corollaries:

\begin{corollary} \label{cor:mainresultLBUB}
	If a join-congruence uniform lattice is shellable, then it is join-extremal. If a meet-congruence uniform lattice is shellable, then it is meet-extremal.
\end{corollary}

\begin{proof}
	This follows from Proposition \ref{PropExtremality} and Theorem \ref{thm:CNextremalshellable}.   
\end{proof}

The following was stated as Theorem \ref{thm:mainthmintro} in the introduction, and contains Theorem \ref{thm:principal}:

\begin{theorem} \label{thm:CUextremalshellable}
	Let $L$ be a congruence uniform lattice. The following properties are equivalent:
	\begin{itemize}
		\item[$(i)$] $L$ is extremal.
		\item[$(ii)$] $L$ is left modular.
		\item[$(iii)$] $L$ is $EL$-shellable.
		\item[$(iv)$] $L$ is shellable.
	\end{itemize}    
\end{theorem}

\begin{proof}
	Corollary \ref{corocongruenceuniformExtremal} gives $(i)\Longleftrightarrow (ii)$. As we recalled in Section \ref{sec:Backgroundext}, we have that $(ii)\Longrightarrow (iii) \Longrightarrow (iv)$ is true for any lattice. It remains to prove $(iv) \Longrightarrow (i)$, which follows from Corollary \ref{cor:mainresultLBUB} since a congruence uniform lattice is both join-congruence uniform and meet-congruence uniform.
\end{proof}

We are now interested in proving Theorem \ref{thm:CNextremalshellable}. 

\begin{definition}
	Let $F$ be a chain of $L$. If $F\cap C=\emptyset$, then $F$ corresponds to a chain in $L[C]$ that we denote by $(F,0)$. Otherwise, let $F\cap C=\{c_1,c_2,\dots,c_k\}$ and we denote by $(F,i)$ the chain 
	$$\big(F\setminus (F\,\cap C) \big)\,\cup \,\{(c_1,0),(c_2,0),\dots,(c_i,0)\} \,\cup \, \{(c_{i},1),(c_{i+1},1),\dots,(c_k,1)\}.$$
\end{definition}

Although it is in conflict with our convention that for $n\in \mathbb Z_{\geq 0}$, we have $[n]=\{1,\dots,n\}$, it turns out to be convenient to define $[0]=\{0\}$.

The following lemma follows from the definition of the doubling construction.

\begin{lemma} \label{lem:doublingordercomplex}
	The facets of $\Delta(L[C])$ are all the $(F,i)$ for any maximal chain $F$ of $L$ and $i\in [\,|F\cap C|\,]$. 
\end{lemma}

For any finite lattice $L$, we define a directed multigraph $FA(L)$, called the \emph{facet adjacency graph} of $L$, whose vertices are the maximal chains of $L$, and we have an edge $F\rightarrow G$ between two such chains $F$ and $G$ if $|F\,\cap G| = |G|-1$. A \emph{source set} of a directed multigraph $G$ is a non-empty subset of the vertices $X\subsetneq V(G)$ such that there are no $y\rightarrow x$ with $y\not\in X$ and $x\in X$. This means that no edges starting outside $X$ are pointing to a vertex of $X$. Two source sets $X$ and $Y$ are disjoint if $X\,\cap Y=\emptyset$. See Figure \ref{fig:congnonextremalFA(L)} for an example of a non-extremal congruence uniform lattice, and Figure \ref{fig:extremalcaselGL} for the case of an extremal congruence uniform lattice. In these figures we write the chains as words without including $\hat{0}$ and $\hat{1}$.

\begin{figure}
	\begin{subfigure}{.5\textwidth}
		\centering
		\begin{tikzpicture}
			\begin{scope}[scale=0.3]
				\draw (0,0)--(4,4)--(4,6)--(0,10)--(-4,6)--(-4,4)--(0,0);
				\draw (2,2)--(-2,8);
				\draw (-2,2)--(2,8);  
				\draw (-2,2) node[below left]{$a$};
				\draw (2,2) node[below right]{$b$};
				\draw (-4,4) node[left]{$c$};
				\draw (0,5) node[right]{$d$};
				\draw (4,4) node[right]{$e$};
				\draw (-4,6) node[left]{$f$};
				\draw (4,6) node[right]{$g$};
				\draw (-2,8) node[above left]{$h$};
				\draw (2,8) node[above right]{$i$};
			\end{scope}
		\end{tikzpicture}
		\caption{A non-extremal congruence\\
			uniform lattice $L$.}
		\label{fig:figALnotextremal}
	\end{subfigure}%
	\begin{subfigure}{.5\textwidth}
		\centering
		\begin{tikzpicture}
			\begin{scope}[scale=1.8]
				\node (A) at (0,0) {$begi$};
				\node (B) at (1,0) {$bdi$};
				\node (C) at (2,0) {$bdh$};
				\node (D) at (1,-1) {$adi$};
				\node (E) at (2,-1) {$adh$};
				\node (F) at (3,-1) {$acfh$};   
				
				\draw[->,>=latex] (A)--(B);
				\draw[->,>=latex] (B) to[bend left=10] (C);
				\draw[->,>=latex] (C) to[bend left=10] (B);
				\draw[->,>=latex] (D) to[bend left=10] (E);
				\draw[->,>=latex] (E) to[bend left=10] (D);
				\draw[->,>=latex] (B) to[bend left=10] (D);
				\draw[->,>=latex] (D) to[bend left=10] (B);
				\draw[->,>=latex] (C) to[bend left=10] (E);
				\draw[->,>=latex] (E) to[bend left=10] (C);
				\draw[->,>=latex] (F)--(E);
			\end{scope}
		\end{tikzpicture}
		\caption{The graph $FA(L)$ from Figure \ref{fig:figALnotextremal}. The sets $\{begi\}$ \\and $\{acfh\}$ are two disjoint source sets.}
		\label{fig:sfig2}
	\end{subfigure}
	\caption{}
	\label{fig:congnonextremalFA(L)}
\end{figure}


\begin{figure}
	\begin{subfigure}{.5\textwidth}
		\centering
		\begin{tikzpicture}
			\begin{scope}[scale=0.5]
				\draw (2,0)--(2,2)--(2,8)--(-4,11)--(-4,7)--(-2,6)--(-2,2)--(2,0);
				\draw (2,2)--(-2,4);
				\draw (2,4)--(-2,6)--(-2,10);
				\draw (2,6)--(-2,8);
				\draw (0,7)--(0,9);    
				\draw (2,2) node[right]{$b$};
				\draw (2,4) node[right]{$d$};
				\draw (2,6) node[right]{$f$};
				\draw (2,8) node[right]{$j$};
				\draw (0,7) node[above right]{$h$};
				\draw (0,9) node[above right]{$k$};
				\draw (-2,2) node[left]{$a$};
				\draw (-2,4) node[left]{$c$};
				\draw (-2,6) node[below left]{$e$};
				\draw (-2,8) node[above right]{$i$};
				\draw (-2,10) node[above right]{$l$};
				\draw (-4,7) node[above right]{$g$};
				\begin{scope}[xshift=6.5cm,yshift=1cm]
					\draw (0,8) node{$A:=a\,c\,e\,i\,l$};
					\draw (-0.1,7) node{$B:=a\,c\,e\,g$};
					\draw (0,6) node{$C:=b\,c\,e\,i\,l$};
					\draw (-0.1,5) node{$D:=b\,c\,e\,g$};
					\draw (0,4) node{$E:=b\,d\,e\,i\,l$};
					\draw (-0.1,3) node{$F:=b\,d\,e\,g$};
					\draw (0.3,2) node{$G:=b\,d\,f\,h\,i\,l$};
					\draw (0.3,1) node{$H:=b\,d\,f\,h\,k\,l$};
					\draw (0.3,0) node{$I:=b\,d\,f\,j\,k\,l$};    
				\end{scope}
			\end{scope}
		\end{tikzpicture}
		\caption{An extremal congruence uniform lattice\\ $L$
			with a labeling of the maximal chains,\\
            omitting the extremum elements.}
		\label{fig:figAextremal}
	\end{subfigure}%
	\begin{subfigure}{.5\textwidth}
		\centering
		\begin{tikzpicture}
			\begin{scope}[scale=1.4]
				\node (I) at (0,0) {$I$};
				\node (H) at (1,0) {$H$};
				\node (G) at (2,0) {$G$};
				\node (E) at (3,0) {$E$};
				\node (C) at (4,0) {$C$};
				\node (A) at (5,0) {$A$}; 
				\node (F) at (3,-1) {$F$}; 
				\node (D) at (4,-1) {$D$};
				\node (B) at (5,-1) {$B$}; 
				
				\draw[->,>=latex] (G)--(E);
				\draw[->,>=latex] (I) to[bend left=10] (H);
				\draw[->,>=latex] (H) to[bend left=10] (I);
				\draw[->,>=latex] (H) to[bend left=10] (G);
				\draw[->,>=latex] (G) to[bend left=10] (H);
				\draw[->,>=latex] (E) to[bend left=10] (C);
				\draw[->,>=latex] (C) to[bend left=10] (E);
				\draw[->,>=latex] (C) to[bend left=10] (A);
				\draw[->,>=latex] (A) to[bend left=10] (C);
				\draw[->,>=latex] (F) to[bend left=10] (D);
				\draw[->,>=latex] (D) to[bend left=10] (F);
				\draw[->,>=latex] (D) to[bend left=10] (B);
				\draw[->,>=latex] (B) to[bend left=10] (D);
				\draw[->,>=latex] (E)--(F);
				\draw[->,>=latex] (C)--(D);
				\draw[->,>=latex] (A)--(B);
			\end{scope}
		\end{tikzpicture}
		\caption{The graph $FA(L)$ from Figure \ref{fig:figAextremal}, whose source \\sets are $\{I,H,G\}$ and $\{I,H,G,E,C,A\}$.}
		\label{fig:figBFA(L)extremal}
	\end{subfigure}
	\caption{}
	\label{fig:extremalcaselGL}
\end{figure}

The following is immediate from the definition of shellability:

\begin{lemma}  \label{lem:lemforsource}
	If $F_1,F_2,\dots ,F_k$ is a shelling of $\Delta(L)$, then for all $2\leq j\leq k$, there exists $i\in \{1,\dots,j-1\}$ such that $F_i \rightarrow F_j$.    
\end{lemma}

We deduce a necessary condition for shellability of $L$ on its facet adjacency graph.

\begin{lemma} \label{lem:shellableonesource}
	If $L$ is shellable, then $FA(L)$ cannot have two disjoint source sets.  
\end{lemma}

\begin{proof}
	Assume that $FA(L)$ has two disjoint source sets $S$ and $T$. Seeking a contradiction, suppose that $L$ is shellable. Thus we have a shelling $F_1,F_2,\dots,F_k$ of $\Delta(L)$. Let $F_l$ and $F_j$ be the left-most maximal chains of respectively $S$ and $T$ in our shelling. Since $F_l\neq F_j$ because the source sets $S$ and $T$ are disjoint, without loss of generality we can assume that $j\geq 2$. Using Lemma \ref{lem:lemforsource}, there exists $i\in [j-1]$ such that $F_i\rightarrow F_j$. Since $F_j$ is the left-most maximal chain of $T$ in our shelling, it implies that $F_i\not\in T$. Since $T$ is a source set, this is absurd. 
\end{proof}

We can obtain $FA(L[C])$ from $FA(L)$:

\begin{proposition}  \label{prop:algoFA(L[C])}
	The vertices of $FA(L[C])$ are the $(F,i)$ for any maximal chain $F$ of $L$ and $i\in [\,|F\cap C|\,]$. We give the edges of $FA(L[C])$ knowing these of $FA(L)$:
	\begin{itemize}
		\item We have $(F,i) \rightarrow (G,0)$ whenever $F\rightarrow G$ in $FA(L)$ with $G\cap C=\emptyset$, where $i\in [\,|F\cap C|\,]$.
		\item For $F\cap C\neq \emptyset$ and $G\cap C\neq \emptyset$, we have $(F,i)\rightarrow (G,j)$ whenever $F\rightarrow G$ in $FA(L)$ and $|(F,i)\cap (G,j)|=|G|$, for any $i\in [\,|F\,\cap C|\,]$ and $j\in [\,|G\,\cap C|\,]$.
		\item We have $(F,i)\rightarrow (F,i+1)$ and $(F,i+1)\rightarrow (F,i)$ whenever $|F\cap C|\geq 2$, for any $i\in [\,|F\cap C|-1]$.
	\end{itemize}
\end{proposition}

\begin{proof}
	We have $|(F,0)|=|F|$ and $|(F,i)|=|F|+1$ if $i\neq 0$. Then the result follows from Lemma \ref{lem:doublingordercomplex} and the definition of the doubling construction.
\end{proof}

\begin{lemma} \label{lem:doubleincreasesources}
	If $FA(L)$ has at least two disjoint source sets, it is also the case for $FA(L[C])$.   
\end{lemma}

\begin{proof}
	Let $S$ and $T$ be two disjoint source sets of $FA(L)$. We define two subsets $S_C:=\{(F,i)\mid F\in S, i\in [\,|F\cap C|\,]\}$ and $T_C:=\{(G,i)\mid G\in T, i\in [\,|G\cap C|\,]\}$ of $FA(L[C])$. We conclude by proving that the sets $S_C$ and $T_C$ are disjoint source sets of $FA(L[C])$. The fact that they are disjoint nonempty proper subsets of $FA(L[C])$ is clear. By Proposition \ref{prop:algoFA(L[C])}, if we do not have $F\rightarrow G$ in $FA(L)$ with $F\neq G$, then we do not have $(F,i)\rightarrow (G,j)$ in $FA(L[C])$. It follows, from this observation and the fact that $S$ and $T$ are source sets, that $S_C$ and $T_C$ are source sets.
\end{proof}

\begin{lemma}  \label{lem:doublingoutsidespine}
	Suppose that $C$ does not intersect the spine of $L$. Then the two sets
	$$S:=\{(F,i)\mid F \text{ is a maximal chain of } L,\, F\cap C\neq \emptyset,\, i\in [\,|F\cap C|\,]\}$$    
	and 
	$$T:=\{(G,0) \mid G \text{ is a longest chain of } L\}$$  
	are two disjoint source sets of $FA(L[C])$.
\end{lemma}

\begin{proof}
	Since for any $(F,i)\in S$ there exists $c\in C$ such that $\{(c,0),(c,1)\} \subseteq (F,i)$, and no maximal chain of $L[C]$ outside of $S$ contains such elements, we know that no edges from a vertex outside $S$ is pointing to $(F,i)$ in $FA(L[C])$. This proves that $S$ is a source set of $FA(L[C])$. 
	
	Seeking a contradiction, suppose that there exist $(F,i)\not\in T$ and $(G,0)\in T$ such that $(F,i)\rightarrow (G,0)$. Since $(F,i)\not\in T$, we have $F\cap C\neq \emptyset$, meaning $i\neq 0$. We have $|(G,0)|=|G|$ and $|(F,i)|=|F|+1$ since $i\neq 0$. In general, if $F\rightarrow F'$ in $FA(L)$, it follows that $|F|\geq |F'|$. Then $(F,i)\rightarrow (G,0) \Longrightarrow |F|+1 \geq |G|$. Since $C$ does not intersect the spine of $L$, by Lemma \ref{lem:lengthdoublement} we have $\ell(L[C])=\ell(L)$. Thus $|G|\leq |F|+1=|(F,i)|\leq \ell(L[C])=\ell(L)=|G|$. 
	Then $|(F,i)|=|G|=\ell(L[C])$, thus $(F,i)$ is a longest chain of $L[C]$. Since $(F,i)\not\in T$, which means $F$ is not a longest chain of $L$, it implies that $(F,i)\in S$ since the chains of longest length of $L[C]$ are either longest chains of $L$ or go through the doubling of $C$.  Then, as observed in the first part of the proof, $(F,i)$ has at least two elements $(c,0)$ and $(c,1)$ not in $(G,0)$, which with $|(F,i)|=|(G,0)|$ makes it absurd to have $|(F,i)\,\cap (G,0)|=|(G,0)|-1$, which is implied by $(F,i)\rightarrow (G,0)$. It follows that $T$ is a source set.
	
	With $S\,\cap T=\emptyset$ since $C$ does not intersect the spine of $L$, this finishes the proof. 
\end{proof}

Now we can prove Theorem \ref{thm:CNextremalshellable}:

\begin{proof}[Proof of Theorem \ref{thm:CNextremalshellable}]
	Let $L=E[C_1,\dots,C_n]$ be a congruence normal lattice. Suppose that there exists $i$ such that $C_i$ does not intersect the spine of $E[C_1,\dots,C_{i-1}]$. Then by Lemma \ref{lem:doublingoutsidespine} the graph $FA(E[C_1,\dots,C_i])$ has at least two disjoint source sets. By Lemma \ref{lem:doubleincreasesources}, since $L=E[C_1,\dots,C_n]$ is obtained from $E[C_1,\dots,C_i]$ by doublings, we have that $L$ has also at least two disjoint source sets, which by Lemma \ref{lem:shellableonesource} implies that $L$ is not shellable.
\end{proof}

\subsection{Induced subgraphs of a canonical join graph}
\label{sec:canonicaljoingraph}

In this section we give, using the fundamental theorem of finite semidistributive lattices \cite{reading2021fundamental}, a counterexample to the following:

\begin{question}[{\cite[Question 2.5.5]{Barnardthesis}}]
\label{questionbarnard}
Is an induced subcomplex of a canonical join complex still a canonical join complex?
\end{question}

As explained in Section \ref{sec:Backgroundext}, for us a canonical join complex is always one coming from a semidistributive lattice (not just a join-semidistributive lattice). Question \ref{questionbarnard} was asked in this context.
The canonical join complex being the clique complex of its 1-skeleton, which is the canonical join graph, this question is equivalent to asking whether an induced subgraph of a canonical join graph is still such a graph. Indeed, an induced subcomplex on vertices $X$ of a flag simplicial complex $\Delta$ is still flag, and its $1$-skeleton is the induced subgraph of the $1$-skeleton of $\Delta$ on vertices $X$.

\begin{figure}
    \centering
\begin{tikzpicture}[every node/.style={circle, draw, minimum size=8mm, inner sep=0pt},scale=1.3]
    \node (v0) at (0,2) {$v_0$};
    \node (v1) at (1.5,2) {$v_1$};
    \node (v2) at (3,2) {$v_2$};
    \node (v3) at (4.5,2) {$v_3$};
    \node (v4) at (6,2) {$v_4$};
    
    \node (u0) at (0,0.5) {$u_0$};
    \node (u1) at (1.5,0.5) {$u_1$};
    \node (u2) at (3,0.5) {$u_2$};
    \node (u3) at (4.5,0.5) {$u_3$};
    \node (u4) at (6,0.5) {$u_4$};
    
    \node (w) at (3,-1) {$w$};
    \node (t) at (3,3.8){$t$};

    \draw[->>] (v2) -- (t);
    \draw[right hook->] (t) -- (v1);     
    
    \draw[<<-] (v0) .. controls (3,5) .. (v4);
    \draw[right hook->] (u0) .. controls (-2,4) and (3,3).. (v4);
    \draw[right hook->>] (u4) .. controls (8,4) and (3,3).. (v0);
    
    \foreach \i/\j in {0/1,1/2,2/3,3/4}
        \draw[<-right hook] (v\i) -- (u\j);
    \foreach \i/\j in {1/0,2/1,3/2,4/3}
        \draw[<-right hook] (v\i) -- (u\j);
        
    \draw[right hook->>] (u1) -- (v0);
    \draw[right hook->>] (u2) -- (v3);
    \draw[right hook->>] (u4) -- (v0);
    \draw[right hook->>] (u4) -- (v3);

    \draw[->>] (v2) -- (v3);
    \draw[->>] (v4) -- (v3);

    \draw[->,>=latex,very thick,red] (v2) -- (v1);
    
    \draw[->>] (v1) -- (v0);

    \foreach \i in {0,1,2,3,4}
        \draw[right hook->>] (u\i) -- (w);
    
\end{tikzpicture}
    \caption{The Galois graph of a semidistributive lattice having $167$ elements. The arrows are of the three types $\rightarrow$, $\hookrightarrow$ and $\twoheadrightarrow$ as in \cite{reading2021fundamental} and they form a two-acyclic factorization system.}
    \label{fig:counterexampleinducedCJC}
\end{figure}
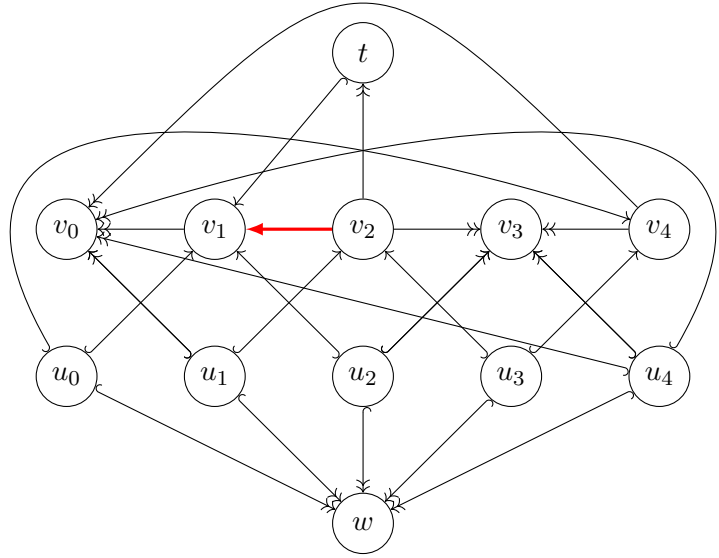

From now on in this section, all directed graphs and multigraphs are assumed to have a loop at every vertex, but they are not represented on the figures. An illustration of the following definitions is given in Figure \ref{fig:counterexampleinducedCJC}.

\begin{definition}
Let $(G,\rightarrow)$ be a directed multigraph. We deduce from it another directed multigraph $G_{A}$ by replacing the arrows $\rightarrow$ with up to three kinds of arrows $\onto, \into, \twoheadrighthookarrow$, where some arrows $\rightarrow$ can remain. An arrow $x\rightarrow y$ is replaced by $x \into y$ if for every $z\in G$ such that $z\rightarrow x$, we have an arrow $z\rightarrow y$. An arrow $x\rightarrow y$ is replaced by $x \onto y$ if for every $z\in G$ such that $y\rightarrow z$, we have an arrow $x\rightarrow z$. If an arrow $x\rightarrow y$ can be replaced by both the arrows $\into$ and $\onto$, then we replace it by $x\twoheadrighthookarrow y$. We understand this latter arrow to be both an $\into \,$-arrow and an $\onto\,$-arrow.
\end{definition}

\begin{definition}[{\cite{reading2021fundamental}}]
\label{def:twoacyclic}
Let $(G,\rightarrow)$ be a directed multigraph. We say that $G$ forms a \emph{two-acyclic factorization system} if $G_A$ satisfies the following conditions:
\begin{enumerate}
    \item[(i)] (order condition) We do not have $x\onto y \onto x$ or $x\into y\into x$ with $x\neq y$.
    \item[(ii)] (brick condition) We do not have $x\onto y\into x$ with $x\neq y$.
    \item[(iii)] (multiplication) For each $x\rightarrow z$ there exists $y$ such that $x\onto y \into z$.
\end{enumerate}
\end{definition}

If $G$ is a multigraph, we say that it admits a \emph{two-acyclic factorization system} if there exists an orientation of its edges such that the resulting directed multigraph $(G,\rightarrow)$ forms a two-acyclic factorization system.
The following is a reformulation of results of \cite{reading2021fundamental}:

\begin{proposition}
\label{prop:cjgmultigraph}
The canonical join graphs are the complements of the multigraphs that admit a two-acyclic factorization system. 
\end{proposition}

\begin{proof}
By \cite[Theorem 1.2]{reading2021fundamental}, a multigraph $G$ admits a two-acyclic factorization system if and only if $G$ can be oriented such that the resulting directed multigraph $(G,\rightarrow)$ gives a semidistributive lattice by taking the associated lattice of maximal orthogonal pairs. By \cite[Theorem 5.13]{reading2021fundamental}, this is equivalent to $\overline{G}$ being a canonical join graph. This concludes. 
\end{proof}

In Figure \ref{fig:counterexampleinducedCJC}, we give an example of the Galois graph $G(L)$ of a semidistributive extremal lattice. Indeed, conditions $(i)$ and $(ii)$ of Definition \ref{def:twoacyclic} are satisfied since there are no cycles, and for condition $(iii)$ we have only one regular arrow $v_2\rightarrow v_1$ in $G(L)_A$; thus, this condition is satisfied since $v_2 \onto t \into v_1$. Let $H:=\overline{G(L)}$. By Proposition \ref{prop:cjgmultigraph}, $H$ is a canonical join graph.

For a graph $G=(V,E)$ and $X\subseteq V$, we denote by $G[X]$ the induced subgraph of $G$ on the elements of $X$.
Let $X:=V(H)\setminus \{t\}$, where $V(H)$ is the vertex-set of $H$. 
Our goal is now to prove that $H[X]$ is not a canonical join graph. Thus it will be a counterexample to Question \ref{questionbarnard}.

The following result is easy:

\begin{lemma}
\label{lem:inducedgraph}
Let $G=(V,E)$ be a multigraph and $X\subseteq V$. We have
$\overline{G[X]} = \overline{G}[X]$.
\end{lemma}

By definition, $\overline{H}=G(L)$. Thus by Lemma \ref{lem:inducedgraph}, we have $\overline{H[X]}=\overline{H}[X] = G(L)[X]$. By Proposition \ref{prop:cjgmultigraph}, $H[X]$ is a canonical join graph if and only if the underlying undirected graph of $G(L)[X]$, or any multigraph obtained from it by adding some edges between two elements of $G(L)[X]$ already joined by an edge, is a two-acyclic factorization system. The underlying undirected graph of $G(L)[X]$ is called the \emph{Grötzsch graph}; it is a triangle-free graph of chromatic number $4$. These two properties will play an important role.

\begin{lemma}
\label{lem:chroma3trianglefree}
Let $G$ be a simple graph. There exists an orientation of the edges of $G$ such that the resulting directed graph has no directed walk of length $3$ if and only if $\chi(G)\leq 3$. 
\end{lemma}

\begin{proof}
Suppose that $\chi(G)\leq 3$. Let us choose a proper coloring $C:V\rightarrow \{0,1,2\}$ of $G$. We deduce an orientation of the edges of $G$ by orienting an edge $x \rightarrow y$ if $C(x)<C(y)$. This is an orientation of $G$ without directed walks of length $3$.
Conversely, suppose that there exists an orientation of the edges of $G$ such that the resulting directed graph $(G,\rightarrow)$ has no directed walk of length $3$. Then $(G,\rightarrow)$ does not have cycles. This proves that the following coloring is well-defined; color any vertex $x$ with the longest length of a directed walk ending at $x$. This gives a coloring $C:V\rightarrow \{0,1,2\}$, which is proper since for each $x\rightarrow y$, we have $C(x)<C(y)$.
\end{proof}

\begin{lemma}
\label{lem:equiv2acyclic3directed}
A triangle-free simple directed graph $G$ forms a two-acyclic factorization system if and only if $G$ does not have a directed walk of length $3$.
\end{lemma}

\begin{proof}
Since $G$ is simple, conditions $(i)$ and $(ii)$ of Definition \ref{def:twoacyclic} are satisfied. On the contrary, condition $(iii)$ cannot be satisfied if there are any $\rightarrow$ arrows since $G$ is triangle-free. Thus $G$ forms a two-acyclic factorization system if and only if $G_A$ does not have regular arrows $\rightarrow$. By the definitions of $\into$ and $\onto$ and the fact that $G$ is triangle-free, we have regular arrows $\rightarrow$ in $G_A$ if and only if $G$ does not have a directed walk of length $3$.
\end{proof}

The following result is a consequence of Lemmas \ref{lem:chroma3trianglefree} and \ref{lem:equiv2acyclic3directed}:

\begin{proposition}
\label{prop:trianglefreeftfsdl}
A triangle-free simple graph $G$ admits a two-acyclic factorization system if and only if $\chi(G)\leq 3$. 
\end{proposition}

\begin{proposition}
\label{prop:trianglefreemultitosimple}
Let $G$ be a triangle-free multigraph. Let $G'$ be its underlying simple graph, which is obtained by replacing multiple edges by a single edge. If $G$ admits a two-acyclic factorization system, then it is also the case of $G'$.   
\end{proposition}

\begin{proof}
Let $(G,\rightarrow)$ be an orientation of $G$ that forms a two-acyclic factorization system. From it we deduce an orientation $(G',\rightarrow')$ of $G'$; for multiple edges between $x$ and $y$ in $G$, choose arbitrarily one of the arrows between $x$ and $y$ in $(G,\rightarrow)$. Let us prove that $(G',\rightarrow')$ forms a two-acyclic factorization system. As seen in the proof of Lemma \ref{lem:equiv2acyclic3directed}, since $(G',\rightarrow')$ does not have $2$-cycles we only need to look at condition $(iii)$ of Definition \ref{def:twoacyclic}. As $G$ and $G'$ are triangle-free, condition $(iii)$ cannot be satisfied if there are any $\rightarrow$ arrows. This means that $G_A$, which is obtained from $(G,\rightarrow)$, does not have regular arrows $\rightarrow$. By definitions of $\into$ and $\onto$, it is also the case of $G'_A$. This proves that $G'$ admits a two-acyclic factorization system.
\end{proof}

Then the underlying undirected graph of $G(L)[X]$, which is triangle-free with a chromatic number equal to $4$, does not admit a two-acyclic factorization system by Proposition \ref{prop:trianglefreeftfsdl}. By Proposition \ref{prop:trianglefreemultitosimple}, any multigraph obtained from the underlying undirected graph of $G(L)[X]$ by adding some edges between two elements of $G(L)[X]$ already joined by an edge, does not admit either a two-acyclic factorization system. This proves that $H[X]$ is not a canonical join graph.

\subsection{Dimension of semidistributive extremal lattices}
\label{sec:dimensionresult}

In this section we prove Theorem \ref{thm:dimresultintro} from the introduction, that gives a way to compute the dimension of a semidistributive extremal lattice.

\begin{definition}
	Let $P$ be a poset. The \emph{critical pairs} of $P$ are the pairs $(a,b)\in P\times P$ such that $a$ and $b$ are incomparable, for any $x<a$ we have $x<b$, and for any $y>b$ we have $y>a$.
\end{definition}

\begin{definition}[\cite{reading2002order}]
	Let $P$ be a poset. The \emph{graph of critical pairs} $D(P)$ associated to $P$ is the directed graph whose vertices are the critical pairs and edges are $(a,b)\rightarrow (c,d)$ whenever $b\geq c$. 
	The \emph{critical pairs simplicial complex} $C(P)$ associated to $P$ is the simplicial complex whose faces correspond to the subsets of critical pairs that induce an acyclic subgraph of $D(P)$ (we mean that we need to avoid directed cycles). 
\end{definition}

A \emph{cover set} of a simplicial complex $\Delta$ is a subset $S$ of the faces of $\Delta$ such that every vertex of $\Delta$ is contained in a face of $S$. The next result follows from \cite[Chapter 4, Proposition 2.1]{trotter2002combinatorics}, and we restate it as it is done in \cite[Proposition 14 and the discussion that follows]{reading2002order}.

\begin{theorem}
\label{thm:readingdimcoversets}
For any poset $P$, the dimension of $P$ is equal to the minimum size of a cover set of $C(P)$. 
\end{theorem}

The following is easy to prove.

\begin{lemma}
	Let $L$ be a semidistributive lattice. Then the critical pairs are exactly the pairs $(j,\kappa(j))$ for $j\in \mathrm{JIrr}(L)$ with $j_{*}\not\in \mathrm{MIrr}(L)$, which means $\kappa(j)\neq j_{*}$. 
\end{lemma}

Thus, for $L$ a semidistributive lattice, we identify a critical pair with the join-irreducible element given by its first coordinate. Then the graph $D(L)$ is the graph on vertices the join-irreducible elements $j$ such that $\kappa(j)\neq j_{*}$, with an edge $i\rightarrow j$ whenever $\kappa(i)\geq j$. Looking at Proposition \ref{prop:kappaextremalSD}, when $L$ is a semidistributive extremal lattice, we see that $D(L)$ and the Galois graph $G(L)$ are similar. We make this observation precise in the next result.

\begin{figure}
	\begin{subfigure}{.5\textwidth}
		\centering
		\begin{tikzpicture}[xscale=0.85,yscale=0.85]
			\node (empty) at (0,0) [circle,fill,inner sep=1.5pt]{};
			\node (1) at (-1,1) {$a$};
			\node (12) at (-2,2) {$b$};
			\node (123) at (-4,4) {$\kappa(c)=g$\hspace*{24pt}};
			\node (124) at (-1,3) {$c$};
			\node (6) at (2,2) {$d$};
			\node (67) at (3,3) {\hspace*{24pt}$e=\kappa(a)$};
			\node (16) at (1,3) {$\kappa(b)$};
			\node (1246) at (0,4) {$\kappa(e)$};
			\node (12467) at (1,5) {$\kappa(f)$};
			\node (1245) at (-2,4) {$f$};
			\node (12345) at (-3,5) {$\kappa(d)$};
			\node (124567) at (0,6) {$\kappa(g)$};
			\node (1234567) at (-1,7) [circle,fill,inner sep=1.5pt]{};
			\draw[-] (empty) -- (1);
			\draw[-] (1) -- (12);
			\draw[-] (12) -- (123);
			\draw[-] (empty) -- (6);
			\draw[-] (6) -- (67);
			\draw[-] (1) -- (16);
			\draw[-] (6) -- (16);
			\draw[-] (16) -- (1246);
			\draw[-] (12) -- (124);
			\draw[-] (124) -- (1246);
			\draw[-] (1246) -- (12467);
			\draw[-] (124) -- (1245);
			\draw[-] (1245) -- (12345);
			\draw[-] (123) -- (12345);
			\draw[-] (12345) -- (1234567);
			\draw[-] (67) -- (12467);
			\draw[-] (12467) -- (124567);
			\draw[-] (124567) -- (1234567);
			\draw[-] (1245) -- (124567);
		\end{tikzpicture}
		\caption{A semidistributive extremal lattice $L$} 
		\label{Pairs fig}
	\end{subfigure}%
	\begin{subfigure}{.5\textwidth}
		\centering
		\begin{tikzpicture}[xscale=1.2,yscale=0.75]
			\node (3) at (0,3) {$g$};
			\node (5) at (-2,2) {$f$};
			\node (7) at (1.5,1) {$e$};
			\node (6) at (2.25,0) {$d$};
			\node (4) at (1.5,-1) {$c$};
			\node (2) at (-2,-2) {$b$};
			\node (1) at (0,-3) {$a$};
			\draw[->] (2) -- (1);
			\draw[->] (3) -- (1);
			\draw[->] (3) -- (2);
			\draw[->] (3) -- (5);
			\draw[->] (3) -- (7);
			\draw[->] (4) -- (1);
			\draw[->] (4) -- (2);
			\draw[->] (5) -- (1);
			\draw[->] (5) -- (2);
			\draw[->] (5) -- (4);
			\draw[->] (5) -- (7);
			\draw[->] (6) -- (4);
			\draw[->] (7) -- (2);
			\draw[->] (7) -- (4);
			\draw[->] (7) -- (6);
			\draw[<->,red,thick] (5) -- (6);
			\draw[<->,red,thick] (3) to[bend left=30] (6);
			\draw[<->,red,thick] (2) -- (6);
			\draw[<->,red,thick] (1) to[bend right=30] (6);
			\draw[<->,red,thick] (1) -- (7);
			\draw[<->,red,thick] (3) -- (7);
		\end{tikzpicture}
		\caption{In black the Galois graph $G(L)$ of the lattice \\ from Figure \ref{Pairs fig}. With the extra 
			red 2-cycles it is $D(L)$.} 
		\label{2afs fig}
	\end{subfigure}
	\caption{}
	\label{fig:DLGL}
\end{figure}

A 2-cycle on two vertices $x$ and $y$ in a directed graph will be represented by $x \longleftrightarrow y$. See Figure \ref{fig:DLGL}, which is partly taken from \cite[Figures 1 and 2]{reading2021fundamental}, for an illustration of the following result in a case where there are no $j$ such that $\kappa(j)=j_*$.

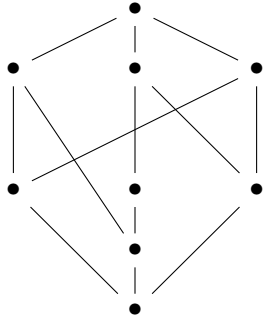
\begin{figure}
    \centering
\begin{tikzpicture}[scale=0.8]
\draw (0,0) node(s0){$\bullet$};  
\draw (0,1) node(s1){$\bullet$};  
\draw (-2,2) node(s2){$\bullet$};  
\draw (0,2) node(s3){$\bullet$};  
\draw (2,2) node(s4){$\bullet$};  
\draw (-2,4) node(s5){$\bullet$};  
\draw (0,4) node(s6){$\bullet$};  
\draw (2,4) node(s7){$\bullet$};  
\draw (0,5) node(s8){$\bullet$}; 

\draw (s0) -- (s1);
\draw (s0) -- (s2);
\draw (s0) -- (s4);
\draw (s1) -- (s3);
\draw (s1) -- (s5);
\draw (s2) -- (s5);
\draw (s2) -- (s7);
\draw (s3) -- (s6);
\draw (s4) -- (s6);
\draw (s4) -- (s7);
\draw (s5) -- (s8);
\draw (s6) -- (s8);
\draw (s7) -- (s8);
\end{tikzpicture}
    \caption{An extremal lattice $L$ that is not semidistributive. }
    \label{fig:counterexamplenotSD}
\end{figure}

\begin{lemma} \label{lem:Galoistocritical}
	Let $L$ be a semidistributive extremal lattice.
	The graph of critical pairs $D(L)$ is obtained from the Galois graph $G(L)$ by adding a $2$-cycle between all pairs of join-irreducible elements $i$ and $j$ that form an independent set of $G(L)$, and then removing the vertices $j$ such that $\kappa(j)=j_{*}$.   
\end{lemma}

\begin{proof}
Let us prove that if $i,j \in \mathrm{JIrr}(L)$ such that $\kappa(i)\neq i_*$ and $\kappa(j)\neq j_*$, then we have $i\rightarrow j$ in $D(L)$ if and only if $j\rightarrow i$ is not in $G(L)$. This follows from the definitions of the arrows in each graph; $i\rightarrow j$ is in $D(L)$ means $\kappa(i)\geq j$, whereas by Proposition \ref{prop:kappaextremalSD}, $j\rightarrow i$ is not in $G(L)$ means $j\leq \kappa(i)$.
Thus if we have $2$-cycles in $G(L)$ they disappear in $D(L)$, but since $G(L)$ is acyclic by Proposition \ref{prop:acyextremal}, $G(L)$ does not have $2$-cycles. The Lemma follows.
\end{proof}

\begin{lemma} \label{lem:noncritiquesnotimportant}
	Let $L$ be a semidistributive extremal lattice. The isolated vertices of $\overline{G(L)}$ are the join-irreducible elements $i\in \mathrm{JIrr}(L)$ such that $\kappa(i)=i_{*}$.
\end{lemma}

\begin{proof}
	Suppose that $i\in \mathrm{JIrr}(L)$ such that $\kappa(i)=i_{*}$. Let us prove that $i$ is an isolated vertex of $\overline{G(L)}$, which means that for every $j\in \mathrm{JIrr}(L)$, $\{i,j\}$ is not an independent set. Seeking a contradiction, assume that $\{i,j\}$ is an independent set of $G(L)$, which means $i\leq \kappa(j)$ and $j\leq \kappa(i)=i_{*}$. This gives $j<\kappa(j)$, which is absurd.
	
	Suppose that $i\in \mathrm{JIrr}(L)$ such that $\kappa(i)\neq i_{*}$. We want to prove that $i$ is not an isolated vertex of $\overline{G(L)}$. Since $L$ is semidistributive and extremal, by Lemma \ref{lem:canonicaljoingraphiscompgalois} the downward label sets of $L$ are the independence sets of its Galois graph $G(L)$. In a semidistributive lattice, the set of all downward label sets is the same as that of all upward label sets, which are denoted by $U(x)$ for an element $x\in L$ \cite[Corollary 6.4]{Thomas_2019}. By definition of $\gamma_J$, we have $i\in U(i_{*})$. Since $\kappa(i)\neq i_{*}$, we have $|U(i_{*})|\geq 2$. Thus $i$ is part of an independence set of size bigger than two, which proves that it is not an isolated vertex of $\overline{G(L)}$.
\end{proof}

For $L$ a semidistributive extremal lattice, we denote by $K(L)$ the undirected graph formed by the subgraph of $D(L)$ obtained by keeping only the $2$-cycles $x\longleftrightarrow y$ of $D(L)$ and replacing these by an edge joining $x$ and $y$.
The \emph{null graph} is the graph that does not have vertices. An \emph{empty graph} is a graph without edges (but it can have vertices).

\begin{proposition} \label{prop:KLiscomplement}
	Let $L$ be a semidistributive extremal lattice. Then $\overline{G(L)}$ is the disjoint union of the graph $K(L)$ with the empty graph whose vertices are the $i\in \mathrm{JIrr}(L)$ such that $\kappa(i)=i_*$.     
\end{proposition}

\begin{proof}
	By Lemma \ref{lem:noncritiquesnotimportant}, we know that the isolated vertices of $\overline{G(L)}$ are the join-irreducible elements $i\in \mathrm{JIrr}(L)$ such that $\kappa(i)=i_*$. Then the proposition follows from Lemma \ref{lem:Galoistocritical}. 
\end{proof}

\begin{proposition}  \label{prop:coversetindset}
Let $L$ be a semidistributive extremal lattice. The cover sets of $C(L)$ are the partitions of the vertices of $K(L)$ into independent sets.  
\end{proposition}

\begin{proof}
By definition the vertices of $C(L)$ and $K(L)$ are the same. 

Let $S$ be a cover set of $C(L)$. We know that $S$ is a partition of the vertices of $K(L)$, and we want to prove that it is a partition into independent sets. Let $F\in S$. By hypothesis $F$ is a face of $C(L)$, which means that the induced subgraph of $D(L)$ on $F$ does not have a directed cycle. In particular it does not contain $2$-cycles, which proves that $F$ is an independent set of $K(L)$.

Conversely, let $S$ be a partition of $K(L)$ into independent sets. We know that it is a partition of the vertices of $C(L)$. Let $F\in S$. By hypothesis $F$ does not contain $2$-cycles. But by Lemma \ref{lem:Galoistocritical} and the fact that $G(L)$ does not have a directed cycle (Proposition \ref{prop:acyextremal}), it follows that the only directed cycles in $D(L)$ are $2$-cycles. Thus the induced subgraph of $D(L)$ on $F$ does not contain directed cycles, which proves that $F$ is a face of $C(L)$. Thus $S$ is a cover set of $C(L)$. 
\end{proof}

We can now prove the main result of this section (which is Theorem \ref{thm:dimresultintro} from the introduction):

\begin{theorem} \label{thm:dimension}
	Let $L$ be a semidistributive extremal lattice with $|L|>1$, then $\dim(L)~=~\chi (\overline{G(L)})$.
\end{theorem}

\begin{proof}
If $L$ is a chain with at least two elements, then the underlying undirected graph of $G(L)$ is the complete graph on $|L|-1$ vertices. Thus $\overline{G(L)}$ is an empty graph with at least one vertex, and the result follows.    
If $L$ is not a chain, then the underlying undirected graph of $G(L)$ is not the complete graph. Thus, by Proposition \ref{prop:KLiscomplement}, $K(L)$ is not the null graph and $\chi(K(L))=\chi(\overline{G(L)})$. Then we conclude using Theorem \ref{thm:readingdimcoversets} and Proposition \ref{prop:coversetindset} which together give $\chi(K(L)) =\dim(L)$.
\end{proof}

\begin{remark}
We defined the Galois graph for any extremal lattice $L$, without the assumption that $L$ is semidistributive. The lattice $L$ from Figure \ref{fig:counterexamplenotSD} shows that it is necessary to have the semidistributivity in Theorem \ref{thm:dimension}. Indeed, $L$ is extremal but not semidistributive, $\dim(L)=3$ and the complement of its Galois graph is a path on $4$ vertices, thus $\chi(\overline{G(L)})=2$.
\end{remark}

\begin{corollary}[{\cite{Dilworth1950Decomposition}}]
\label{cor:dilworth}
Let $L$ be a distributive lattice. Then $\dim(L)=\mathrm{width}(\mathrm{JIrr}(L))$.    
\end{corollary}

\begin{proof}
By the fundamental theorem of finite distributive lattices, the Galois graph $G(L)$ is isomorphic to the Hasse diagram of the subposet $\mathrm{JIrr}(L)$. Indeed, we have an edge $i\rightarrow j$ in $G(L)$ if and only if $i\geq j$ in the subposet $\mathrm{JIrr}(L)$. Thus $\overline{G(L)}$ is isomorphic to the incomparability graph of $\mathrm{JIrr}(L)$. By Dilworth's decomposition theorem, the chromatic number of this latter graph equals the maximum size of an antichain, which is $\mathrm{width}(\mathrm{JIrr}(L))$. We conclude using Theorem \ref{thm:dimension}.
\end{proof}

For a finite simple graph $G$, we denote by $\omega(G)$ its clique number and by $d(G)$ its maximum degree. The following is well-known and easy:

\begin{lemma}
	\label{lem:graphchro}
	For any simple graph $G$, we have $\omega(G) \leq \chi(G) \leq d(G)+1$.
\end{lemma}

The number of independent sets of $G(L)$ is the number of elements of $L$ (\cite[Corollary 5.6]{Thomas_2019}). This corresponds to the number of cliques of $\overline{G(L)}$. By Lemma \ref{lem:canonicaljoingraphiscompgalois}, a clique of $\overline{G(L)}$ corresponds to the downward label set, or upward label set using \cite[Corollary 6.4]{Thomas_2019}, of some element of $L$. Thus $\omega(\overline{G(L)})$ is the maximum number of elements that cover, or are covered, by some vertex of $L$. This is a lower bound for the dimension of $L$ by Theorem \ref{thm:dimension} and Lemma \ref{lem:graphchro}. We obtained differently this result in Proposition \ref{prop:dimSDbigger}, proving it is true for any semidistributive lattice.

\section{Applications of our dimension result}
\label{sec:applicationdimension}

In this section, we use Theorem \ref{thm:dimension} to obtain the dimension of some semidistributive extremal lattices that appeared recently in algebraic combinatorics.

\subsection{Hochschild and bubble lattices}
\label{sec:hochschildbubble}

In this subsection we determine the dimension of the Hochschild lattice \cite{Combe2020HochschildLattices,M_hle_2022} and of one of its generalizations called the bubble lattices \cite{mcconville2024bubble,mcconville2025bubble}. The Hochschild lattice is obtained as an orientation of the 1-skeleton of the Hochschild polytopes, which appeared in algebraic topology \cite{rivera2018combinatorial}.

Recall the discussion on extremal lattices from Section \ref{sec:extremallattices}, in particular Theorem \ref{thm:extremalrepthm}, which enables us to define an extremal lattice by giving its Galois graph.

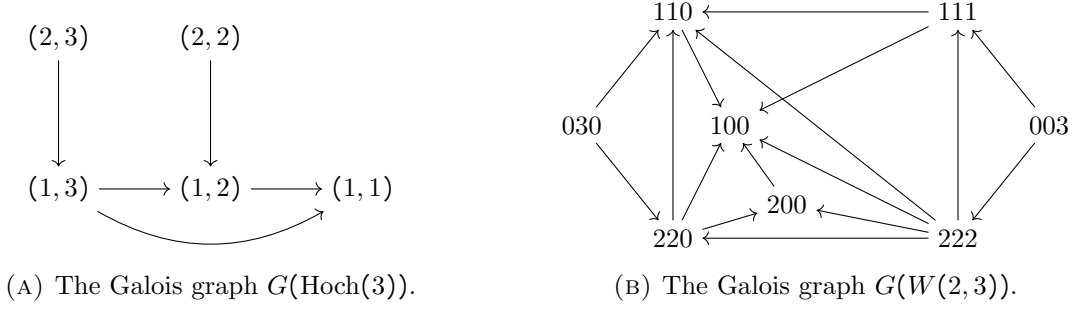
\begin{figure}
	\centering
	\begin{subfigure}[t]{.5\textwidth}
		\centering
		\begin{tikzpicture}\small
			\def\x{1};
			\def\y{1};
			\draw(0*\x,2*\y) node(n1){$(2,3)$};
			\draw(0*\x,0*\y) node(n2){$(1,3)$};
			\draw(2*\x,0*\y) node(n3){$(1,2)$};
			\draw(4*\x,0*\y) node(n4){$(1,1)$};
			\draw(2*\x,2*\y) node(n5){$(2,2)$};
			\draw[->](n1) -- (n2);
			\draw[->](n2) -- (n3);
			\draw[->](n2) to[bend right] (n4);
			\draw[->](n3) -- (n4);
			\draw[->](n5) -- (n3);
		\end{tikzpicture}
		\caption{The Galois graph $G(\mathrm{Hoch}(3))$.}
		\label{fig:galoisHoch3}
	\end{subfigure}%
	\begin{subfigure}[t]{.5\textwidth}
		\centering
\begin{tikzpicture}[scale=1.5]\small
			\def\x{1};
			\def\y{1};
			\draw(0*\x,0*\y) node(n1){$220$};
			\draw(-0.8*\x,1*\y) node(n2){$030$};
			\draw(0.5*\x,1*\y) node(n3){$100$};
			\draw(0*\x,2*\y) node(n4){$110$};
			\draw(1*\x,0.3*\y) node(n5){$200$};
            \draw(2.5*\x,0*\y) node(n6){$222$};
            \draw(2.5*\x,2*\y) node(n7){$111$};
            \draw(3.3*\x,1*\y) node(n8){$003$};
			\draw[->](n1) -- (n3);
                \draw[->](n1) -- (n4);
                \draw[->](n1) -- (n5);
                \draw[->](n2) -- (n1);
                \draw[->](n2) -- (n4);
                \draw[->](n4) -- (n3);
                \draw[->](n5) -- (n3);
                \draw[->](n6) -- (n1);
                \draw[->](n6) -- (n5);
                \draw[->](n6) -- (n3);
                \draw[->](n6) -- (n4);
                \draw[->](n6) -- (n7);
                \draw[->](n7) -- (n3);
                \draw[->](n7) -- (n4);
                \draw[->](n8) -- (n6);
                \draw[->](n8) -- (n7);
		\end{tikzpicture}
		\caption{The Galois graph $G(W(2,3))$.}
		\label{fig:galoisW23}
	\end{subfigure}
	\caption{Two Galois graphs.}
	\label{fig:twogaloisgraphs}
\end{figure}

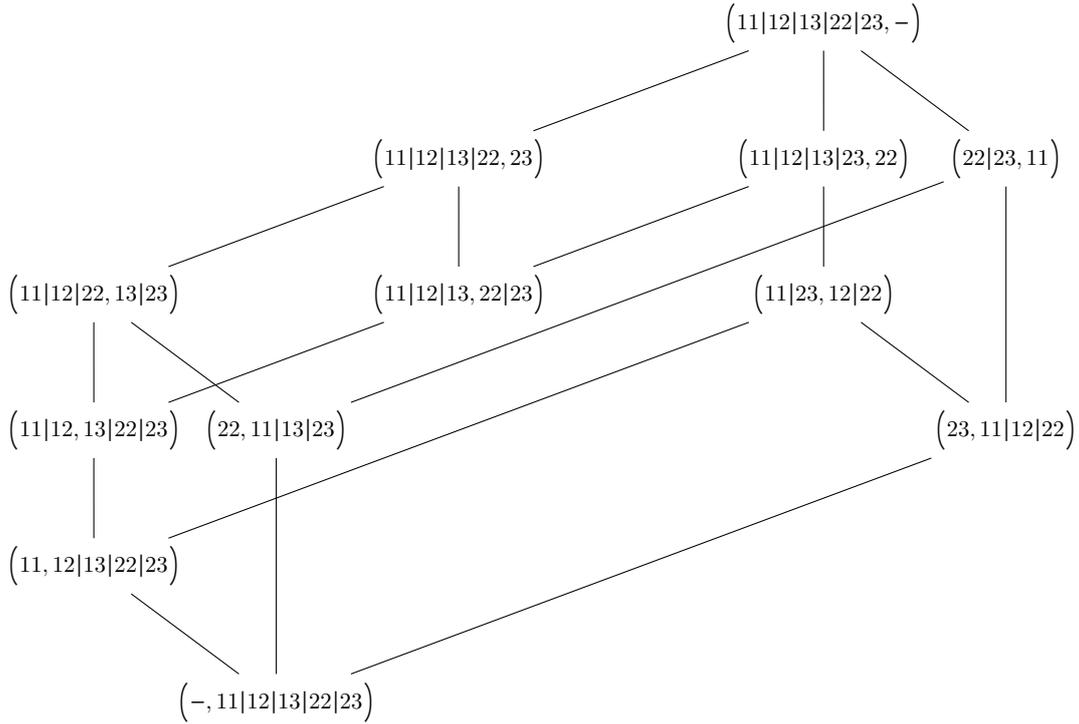
\begin{figure}
	\centering
	\begin{tikzpicture}\small
		\begin{scope}[scale=1.2]
			\def\x{2};
			\def\y{1.5};
			\def\s{.85};
			\draw(3*\x,1*\y) node[scale=\s](n1){$\Bigl(-,11|12|13|22|23\Bigr)$};
			\draw(2*\x,2*\y) node[scale=\s](n2){$\Bigl(11,12|13|22|23\Bigr)$};
			\draw(2*\x,3*\y) node[scale=\s](n3){$\Bigl(11|12,13|22|23\Bigr)$};
			\draw(3*\x,3*\y) node[scale=\s](n4){$\Bigl(22,11|13|23\Bigr)$};
			\draw(7*\x,3*\y) node[scale=\s](n5){$\Bigl(23,11|12|22\Bigr)$};
			\draw(2*\x,4*\y) node[scale=\s](n6){$\Bigl(11|12|22,13|23\Bigr)$};
			\draw(4*\x,4*\y) node[scale=\s](n7){$\Bigl(11|12|13,22|23\Bigr)$};
			\draw(6*\x,4*\y) node[scale=\s](n8){$\Bigl(11|23,12|22\Bigr)$};
			\draw(4*\x,5*\y) node[scale=\s](n9){$\Bigl(11|12|13|22,23\Bigr)$};
			\draw(6*\x,5*\y) node[scale=\s](n10){$\Bigl(11|12|13|23,22\Bigr)$};
			\draw(7*\x,5*\y) node[scale=\s](n11){$\Bigl(22|23,11\Bigr)$};
			\draw(6*\x,6*\y) node[scale=\s](n12){$\Bigl(11|12|13|22|23,-\Bigr)$};
			\draw(n1) -- (n2);
			\draw(n1) -- (n4);
			\draw(n1) -- (n5);
			\draw(n2) -- (n3);
			\draw(n2) -- (n8);
			\draw(n3) -- (n6);
			\draw(n3) -- (n7);
			\draw(n4) -- (n6);
			\draw(n4) -- (n11);
			\draw(n5) -- (n8);
			\draw(n5) -- (n11);
			\draw(n6) -- (n9);
			\draw(n7) -- (n9);
			\draw(n7) -- (n10);
			\draw(n8) -- (n10);
			\draw(n9) -- (n12);
			\draw(n10) -- (n12);
			\draw(n11) -- (n12);
		\end{scope}
	\end{tikzpicture}
	\caption{The lattice of maximal orthogonal pairs of $G(\mathrm{Hoch}(3))$ from Figure \ref{fig:galoisHoch3}, where pairs $(i,j)$ are abbreviated by the words $ij$ and we omitted set brackets and replaced commas by vertical bars.}
	\label{fig:hoch3}
\end{figure}

\begin{figure}
	\centering
	\includegraphics[width=8cm]{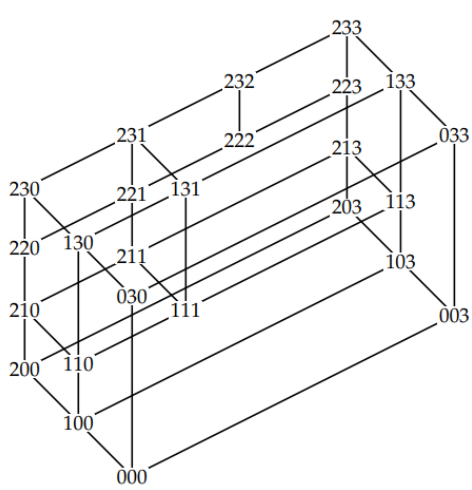}
	\caption{The lattice $W(2,3)$.}
	\label{fig:latticeW23}
\end{figure}

\begin{definition}[{\cite[Theorem 1.1]{M_hle_2022}}]
	\label{def:Hochlattice}
	For any $n>0$, the \emph{Hochschild lattice} $\mathrm{Hoch}(n)$ is an extremal lattice whose Galois graph is isomorphic to the directed graph whose vertices are $\{(i,j)\mid i\in \{1,2\},\,i\leq j\leq n\}$ and the edges are $(i,j)\rightarrow (i',j')$ if $(i,j)\neq (i',j')$ and either we have $i=2$, $i'=1$ and $j=j'$, or we have $i=i'=1$ and $j>j'$.
\end{definition}

We give in Figure \ref{fig:galoisHoch3} the Galois graph of $\mathrm{Hoch}(3)$, and we give this lattice in Figure \ref{fig:hoch3} (this figure is taken from \cite{M_hle_2022}).

\begin{definition}[{\cite[Proposition 4.27]{mcconville2024bubble}}]
	Let $X=\{x_1,\dots,x_m\}$ and $Y=\{y_1,\dots,y_n\}$ be two disjoint linearly ordered alphabets.
	For any $m,n\geq 0$, the \emph{bubble lattice} $\mathrm{Bub}(m,n)$ is an extremal lattice whose Galois graph is the directed graph whose vertices are $X\sqcup Y\sqcup (X\times Y)$ and the edges are $l_1\rightarrow l_2$ if $l_1\neq l_2$ and either:
	\begin{enumerate}
		\item[$\bullet$] $l_1=(x_s,y_t)$ and $l_2=x_s$, or
		\item[$\bullet$] $l_1=y_t$ and $l_2=(x_s,y_t)$, or
		\item[$\bullet$] $l_1=(x_s,y_t)$ and $l_2=(x_{s'},y_{t'})$ with $s\geq s'$ and $t\leq t'$.
	\end{enumerate}
\end{definition}

It is proved in \cite[Proposition 4.19]{mcconville2024bubble} that $\mathrm{Bub}(n-1,1)$ is isomorphic to $\mathrm{Hoch}(n)$. Moreover, these lattices are semidistributive and extremal \cite[Theorem 1.1]{mcconville2024bubble}.

Now we use Theorem \ref{thm:dimension} to obtain the dimension of $\mathrm{Hoch}(n)$. 

\begin{proposition}
	\label{prop:dimHoch}
	We have $\dim(\mathrm{Hoch}(n))=n$ for all $n>0$. 
\end{proposition}

\begin{proof}
	From the Galois graph of $L:=\mathrm{Hoch}(n)$ given in Definition \ref{def:Hochlattice}, we can easily deduce $\overline{G(L)}$. Let $A:=\{(1,i)\mid i\in [n]\}$ and $B:=\{(2,i)\mid 2\leq i\leq n\}$. Then $\overline{G(L)}$ is the graph whose vertices are $A\sqcup B$, and it can be described by the fact that the induced subgraph on $A$ is an empty graph, and for every $i\geq 2$ the induced subgraph on $\{(1,i)\}\sqcup B$ is the complete graph on $n$ vertices. From the latter we obtain $\chi(\overline{G(L)})\geq n$. Moreover, by coloring the vertices of $\{(1,1)\}\sqcup B$ with $n$ different colors, and coloring the vertices of $A$ all with the same color as $(1,1)$, this gives a proper coloring with $n$ colors. Thus $\chi(\overline{G(L)})= n$, which implies the proposition, by Theorem \ref{thm:dimension}.
\end{proof}

In fact, we can directly obtain the dimension of the bubble lattices (which is \cite[Conjecture 4.24]{mcconville2024bubble}) using the bounds on dimension from Proposition \ref{prop:bounddimreading}, which also implies Proposition \ref{prop:dimHoch}.

\begin{proposition}
\label{prop:dimbubble}
We have $\dim(\mathrm{Bub}(m,n))=m+n$ for all $m,n\geq 0$.
\end{proposition}

\begin{proof}
	In \cite[Corollary 4.10]{mcconville2024bubble}, it is proved that the subposet $\mathrm{JIrr}(\mathrm{Bub}(m,n))$ on the join-irreducible elements of $\mathrm{Bub}(m,n)$ is the disjoint union of an $m$-element antichain with $n$ chains. This proves that $\mathrm{Bub}(m,n)$ has $m+n$ atoms and $\mathrm{width}\big(\mathrm{JIrr}(\mathrm{Bub}(m,n))\big)=m+n$, which is an upper-bound of $\dim(\mathrm{Bub}(m,n))$ by Proposition \ref{prop:bounddimreading}. Since $\mathrm{Bub}(m,n)$ is a semidistributive lattice with $m+n$ atoms, using Proposition \ref{prop:dimSDbigger} we have $\dim(\mathrm{Bub}(m,n))\geq m+n$. This implies the proposition.
\end{proof}

\subsection{$(m,n)$-word lattices}
\label{sec:mnwordlattices}

We focus in this section on another generalization of the Hochschild lattice, called $(m,n)$-word lattices (\cite[Definition 77]{pilaud2025hochschild},\cite{Muehle2024WordLattices}).

\begin{definition}[{\cite[Definition 77]{pilaud2025hochschild}}]
	Let $m,n\geq 0$. The $(m,n)$-words are the words $w=w_1\cdots w_n$ of length $n$ on the alphabet $\{0,1,\dots,m+1\}$ such that $w_1\neq m+1$ and for $s\in [m]$, $w_i=s$ implies $w_j\geq s$ for all $j<i$. The \emph{$(m,n)$-word poset} $W(m,n)$ is the poset on $(m,n)$-words with the product order, which means $w\leq w'$ if $w_i\leq w'_i$ for all $i\in [n]$.
\end{definition}

The poset $W(m,n)$ is a lattice \cite[Corollary 84]{pilaud2025hochschild} and is semidistributive and extremal \cite[Theorem 1.1]{Muehle2024WordLattices}. Similarly to the bubble lattice $\mathrm{Bub}(m,n)$, it is isomorphic to $\mathrm{Hoch}(n)$ when $m=1$ \cite[Example 78]{pilaud2025hochschild}. But these are different generalizations of the Hochschild lattice. See Figure \ref{fig:galoisW23} for the Galois graph of $W(2,3)$, and Figure \ref{fig:latticeW23} for the lattice $W(2,3)$.

\begin{lemma}[{\cite[Lemma 12]{Muehle2024WordLattices}}]
	The join-irreducible elements of $W(m,n)$ are the $(m,n)$-words $a^{(i,j)}$ for $i\in [n]$ and $j\in[m]$, and $b^{(i)}$ for $i\in \{2,\dots,n\}$, where $a^{(i,j)}$ is the $(m,n)$-word with the first $i$ letters equal to $j$ and other letters are $0$, and $b^{(i)}$ is the $(m,n)$-word whose $i$-th letter is $m+1$ and all other letters are $0$.  
\end{lemma}

\begin{proposition}[{\cite[Proposition 22]{Muehle2024WordLattices}}]
	\label{prop:canonicaljoinwords}
	The downward label set of $w\in W(m,n)$ is 
    
    \noindent $\{a^{(i,j)} \mid \text{$i$ is the position of the rightmost letter $j$ in $w$, where $j\in [m]$}\} \sqcup \{b^{(i)}\mid w_i=m+1\}$.
\end{proposition}

\begin{proposition}
We have $\dim(W(m,n))=n$ for all $m,n\geq 0$.
\end{proposition}

\begin{proof}
	Let $m,n\geq 0$ and let $L:=W(m,n)$. By Theorem \ref{thm:dimension}, we have $\dim(L)=\chi(\overline{G(L)})$, which is the chromatic number of the canonical join graph of $L$. Color a join-irreducible element of $L$ by the position of its rightmost non-zero letter. By Proposition \ref{prop:canonicaljoinwords}, this is a proper coloring with $n$ colors of $\overline{G(L)}$, thus $\dim(L)\leq n$. As the downward label set of the maximum element $m(m+1)\dots (m+1)$ of $L$ has $n$ elements, thus $\chi(\overline{G(L)})\geq n$, and the proposition follows. 
\end{proof}

\begin{remark}
	We saw that the bounds from Proposition \ref{prop:bounddimreading} were enough to find the dimension of the bubble lattices. For $W(m,n)$ it is not enough as \cite[Lemma 24]{Muehle2024WordLattices} gives its subposet on its join-irreducible elements and it follows that $\mathrm{width}\big( \mathrm{JIrr}(W(m,n)) \big)=n+1$, which is strictly bigger than $\dim(W(m,n))=n$.
\end{remark}

\subsection{Parabolic Tamari lattices}
\label{sec:parabolictamari}

In this section, we are interested in the parabolic Tamari lattice \cite{muhleWilliams,M_hle_2021}, which is a generalization of the Tamari lattice.

\begin{figure}
    \centering
    \includegraphics[width=8cm]{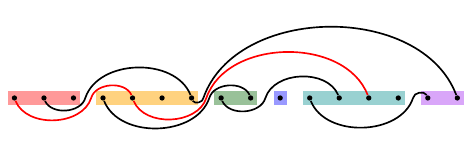}
    \caption{A noncrossing $\alpha$-partition.}
    \label{fig:arcs}
\end{figure}

\begin{figure}
	\centering
	\begin{subfigure}[t]{.5\textwidth}
		\centering
		\includegraphics[width=5cm]{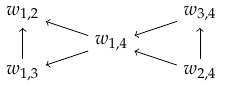}
		\caption{The Galois graph $G(\mathrm{Tam}(1,2,1))$.}
		\label{fig:GaloisgraphTam121}
	\end{subfigure}%
	\begin{subfigure}[t]{.5\textwidth}
		\centering
\begin{tikzpicture}
\node (0) at (0,0) {$\bullet$};
\node (1) at (-1,1) {$\bullet$};
\node (2) at (0,1) {$\bullet$};
\node (3) at (-2,2) {$\bullet$};
\node (4) at (-1,2) {$\bullet$};
\node (5) at (1,2) {$\bullet$};
\node (6) at (-3,3) {$\bullet$};
\node (7) at (-3,4) {$\bullet$};
\node (8) at (-1,4) {$\bullet$};
\node (9) at (-2,5) {$\bullet$};

\path (0) edge[] (1);
\path (0) edge[] (2);
\path (1) edge[] (3);
\path (1) edge[] (4);
\path (2) edge[] (4);
\path (2) edge[] (5);
\path (3) edge[] (6);
\path (4) edge[] (7);
\path (4) edge[] (8);
\path (5) edge[] (8);
\path (6) edge[] (7);
\path (7) edge[] (9);
\path (8) edge[] (9);
\end{tikzpicture}
		\caption{The lattice $\mathrm{Tam}(1,2,1)$.}
		\label{fig:Tam121}
	\end{subfigure}
	\caption{The parabolic Tamari lattice $\mathrm{Tam}(1,2,1)$.}
	\label{fig:parabolictam121}
\end{figure}
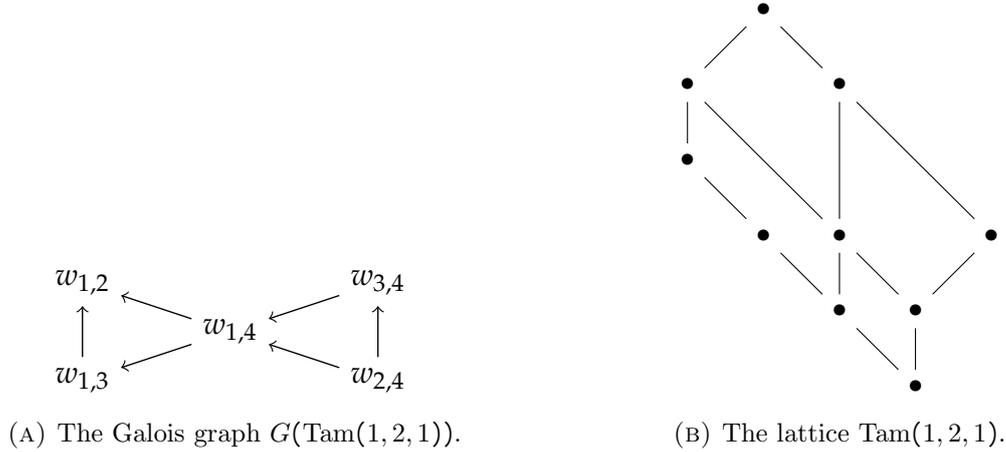

 Let $\alpha=(\alpha_1,\dots,\alpha_k)$ be an integer composition of $n>0$. We draw $n$ horizontally aligned nodes labeled $1,\dots,n$ and put vertical bars separating the node $\alpha_1+\cdots +\alpha_i$ from the node $\alpha_1+\cdots +\alpha_i+1$ for all $i\in [k]$. The sets of nodes between two such separating bars are the $\alpha$-regions, they are naturally labeled $1,\dots,k$ from left to right. 

An $\alpha$-arc is a pair $(a,b)$, where $1\leq a<b\leq n$ and $a,b$ are in different $\alpha$-regions. We represent graphically an $\alpha$-arc $(a,b)$ by drawing a curve leaving the bottom of node $a$, staying below the $\alpha$-region containing $a$, moving up and above the following $\alpha$-regions until it reaches the top of node $b$. We will also denote an $\alpha$-arc $(a,b)$ by $w_{a,b}$.
Two $\alpha$-arcs $(a_1,b_1)$ and $(a_2,b_2)$ are \emph{compatible} if $a_1\neq a_2$, $b_1\neq b_2$, and if $a_1<a_2<b_1<b_2$ then either $a_1,a_2$ or $a_2,b_1$ are in the same $\alpha$-region, and if $a_1<a_2<b_2<b_1$ then $a_1,a_2$ are in the same $\alpha$-region. A set of pairwise compatible $\alpha$-arcs are called \emph{noncrossing $\alpha$-partitions}. The set of all noncrossing $\alpha$-partitions is denoted by $\mathrm{Nonc}(\alpha)$. See Figure \ref{fig:arcs} (which is taken from \cite{M_hle_2021}) for an illustration of a noncrossing $\alpha$-partition, where $\alpha$-regions are depicted with different colors. In this figure $\alpha=(3,4,2,1,4,2)$ and the two red $\alpha$-arcs are $(1,5)$ and $(5,13)$.

\begin{definition}[{\cite[Theorem 1.3]{M_hle_2021}}]
	Let $\alpha$ be an integer composition of $n>0$. The \emph{parabolic Tamari lattice} $\mathrm{Tam}(\alpha)$ is an extremal lattice whose Galois graph is isomorphic to the directed graph whose vertices are the $\alpha$-arcs $w_{a,b}$ and the edges are between different $\alpha$-arcs $w_{a_1,b_1} \rightarrow w_{a_2,b_2}$ such that either $a_1$ and $a_2$ are in the same $\alpha$-region and $a_1\leq a_2<b_2\leq b_1$, or $a_1$ and $a_2$ are in different $\alpha$-regions and $a_2< a_1<b_2\leq b_1$, where $a_1$ and $b_2$ are also in different $\alpha$-regions.
\end{definition}

See Figure \ref{fig:parabolictam121} for an example.
This lattice was defined in \cite{muhleWilliams}. The fact that they are semidistributive and extremal is \cite[Theorem 1.1]{M_hle_2021}. If $\alpha=(1,1,\dots,1)$ then $\mathrm{Tam}(\alpha)$ is the Tamari lattice \cite[Proposition 8]{muhleWilliams}.

Using \cite[Proposition 3.5]{M_hle_2021}, it follows that the canonical join graph of $\mathrm{Tam}(\alpha)$, which is $\overline{G(\mathrm{Tam}(\alpha))}$, is the graph whose vertices are the $\alpha$-arcs, with an edge between two compatible such arcs. Thus its cliques are the noncrossing $\alpha$-partitions.

\begin{proposition}
	\label{prop:dimparabolicTamari}
	Let $\alpha=(\alpha_1,\dots,\alpha_k)$ be an integer composition of $n>0$. Denote by $\max(\alpha)$ the biggest integer appearing in $\alpha$. We have $\dim(\mathrm{Tam}(\alpha))=n-\max(\alpha)$.    
\end{proposition}

\begin{proof}
	Let $i$ such that $\alpha_i=\max(\alpha)$.
	A proper coloring of $\overline{G(\mathrm{Tam}(\alpha))}$ with $n-\max(\alpha)$ colors can be obtained by coloring an $\alpha$-arc $(a,b)$ by $a$ if $a$ is in an $\alpha$-region labeled by $k<i$, and the other $\alpha$-arcs $(a',b')$ which have not yet been assigned a color are colored by $b'$. Thus the colors are labels of nodes that are not in the $i$-th $\alpha$-region, thus this uses at most (in fact exactly) $n-\max(\alpha)$ colors. This is a proper coloring since two $\alpha$-arcs having the same color share the same starting node, or ending node, thus they are not compatible. This proves that $\chi(\overline{G(\mathrm{Tam}(\alpha))})\leq n-\max(\alpha)$.
	
	We still need to prove that $\chi(\overline{G(\mathrm{Tam}(\alpha))})\geq n-\max(\alpha)$. This comes from the fact that we can find a noncrossing $\alpha$-partition with this number of $\alpha$-arcs; for every $r\in [k]$ draw an $\alpha$-arc between the $j^{th}$ node of the $r$-th $\alpha$-region and the $j^{th}$ node of the $(r+1)$-st $\alpha$-region whenever these nodes exist. This forms a noncrossing $\alpha$-partition with $n-\max(\alpha)$ total number of $\alpha$-arcs.
\end{proof}

For example, the planar lattice $\mathrm{Tam}(1,2,1)$ represented in Figure \ref{fig:parabolictam121} is of dimension $2$.

\subsection{On some lattices of torsion classes}
\label{sec:torsionclassesgentletree}

In this section we will mainly focus on the lattices of torsion classes of a gentle tree, but we also include a short discussion about the lattices of torsion classes of the simply laced Dynkin quivers $\mathbf{A}$, $\mathbf{D}$ and $\mathbf{E}$.

\begin{proposition}[{\cite[Theorem 2.10]{Thomas_2019}}]
\label{prop:Arepfinitetorsionclass}
	If $A$ is representation finite and $\mathrm{mod}\,A$ has no cycles, then $\mathrm{Tors}(A)$ is a finite semidistributive extremal lattice. 
\end{proposition}

This proposition applies to the simply laced Dynkin quivers $\mathbf{A},\,\mathbf{D},\,\mathbf{E}$. Other examples, which contain type $\mathbf{A}$, are the gentle trees (see Section \ref{sec:representationtheoryalgebra}). Let us first focus on the former.

Denote by $\mathrm{Camb}(\vec{Q})$ the Cambrian lattice associated to an orientation of a simply laced Dynkin diagram $Q$ (\cite{reading2006cambrian}). It is proved in \cite[Section 4.2]{Ingalls_Thomas_2009} that if $\vec{Q}$ is a simply laced Dynkin quiver, then $\mathrm{Tors}(\mathbb{K} \vec{Q}) \cong \mathrm{Camb}(\vec{Q})$. For example, it is known that when the path quiver $\mathbf{A}$ is linearly oriented, we obtain the Tamari lattice. See also Figure \ref{fig:latticetorsion} for another example, which is a Cambrian lattice of type $\mathbf{A}_3$. 

If $\vec{Q}$ has $n$ vertices, then $\dim(\mathrm{Camb}(\vec{Q})) \geq n$ using Proposition \ref{prop:dimSDbigger}. Also, by definition, $\mathrm{Camb}(\vec{Q})$ is a subposet of the weak order of the Coxeter group associated to $Q$. In \cite{READING2003265}, it is proved that the dimension of the weak order of type $\mathbf{A}_n$ and $\mathbf{D}_n$ is $n$, and bounds are given for the dimension of the weak orders on the exceptional types $\mathbf{E}_6,\,\mathbf{E}_7,\,\mathbf{E}_8$. It follows that for all orientations of the corresponding Coxeter diagrams, $\dim \big(\mathrm{Camb}(\overrightarrow{\mathbf{A}_n})\big)=n$ and $\dim \big(\mathrm{Camb}(\overrightarrow{\mathbf{D}_n}) \big)=n$. For the exceptional types, with the help of the computer, we obtain:

\begin{proposition}
\label{prop:computercambrian}
For all orientations of the corresponding Coxeter diagrams, we have $\dim \big(\mathrm{Camb}(\overrightarrow{\mathbf{E}_6})\big)=6$, $\dim \big(\mathrm{Camb}(\overrightarrow{\mathbf{E}_7})\big)=7$ and $\dim \big(\mathrm{Camb}(\overrightarrow{\mathbf{E}_8})\big)=8$.   
\end{proposition}

\noindent 
{\it Proof (using a computer).}
The computations were performed on the Nibi supercomputer, operated by Calcul Québec and the Digital Research Alliance of Canada (Alliance).
The Cambrian lattices of any orientation of types $\mathbf{E}_6$, $\mathbf{E}_7$, and $\mathbf{E}_8$ have respectively $833$, $4160$, and $25080$ elements. The number of vertices of their respective Galois graphs are $36$, $63$ and $120$. With SageMath, using Theorem \ref{thm:dimension}, we were able to exhaustively check that the Cambrian lattices of the $32$ different orientations of $\mathbf{E}_6$ have all dimension $6$, the Cambrian lattices of the $64$ different orientations of $\mathbf{E}_7$ have all dimension $7$, and the Cambrian lattices of the $128$ different orientations of $\mathbf{E}_8$ have all dimension $8$. The latter required roughly one week of computation.  \qed

\begin{remark}
\label{rmk:cambrian}
For the Cambrian lattices of type $\mathbf{A}$, it is also possible to use a similar combinatorial model to the above discussion about the parabolic Tamari lattice in terms of noncrossing partitions (\cite[Example 4.9]{readingcanonical}, \cite[section 8.3]{barnard2023pop}). In this model, by coloring an arc with the label of its starting node, we prove analogously to Proposition \ref{prop:dimparabolicTamari} that $\dim \big(\mathrm{Camb}(\overrightarrow{\mathbf{A}_n})\big)=n$ for any orientation of $\mathbf{A}_n$.
\end{remark}

Now we focus on the torsion classes of a gentle tree. To see more details about the combinatorics of gentle algebras, see \cite{brustle2020combinatorics,palu2021non}. 
Given a gentle tree $T$, there is an (almost) canonical way to draw it so that arrows always go horizontally to the right or vertically upwards, and such that the relations correspond to a change of direction. (It is almost canonical because for a given arrow, we can choose whether to draw it horizontally or vertically; this forces the direction of all other arrows). From now on, we assume that our gentle trees are drawn in this way (having fixed an orientation of the edges). The horizontal arrows can be thought of as east steps and the vertical ones as north steps. Thus, the relations correspond to taking an east step followed by a north step, or a north step followed by an east step. See Figure \ref{fig:gentletreeexample} for an example of a gentle tree drawn with our conventions.

The indecomposable modules of $T$ correspond to the
walks in the tree which take the unique shortest path in the underlying graph between their two endpoints without containing two steps that form a relation. Those are called the \emph{string modules} and we write $C_\sigma$ for the indecomposable module corresponding to the walk $\sigma$. As a representation this corresponds to putting $\mathbb{K}$ at every vertex visited by $\sigma$, $0$ at every other vertex and the identity map id between any two adjacent copies of $\mathbb{K}$. With the orientation of the gentle tree fixed as above, there is a canonical way to read a walk $\sigma$: from top left to bottom right. Each walk therefore has a well-defined top left endpoint, and a well-defined bottom right endpoint. We refer to them simply as the \emph{left endpoint} and the \emph{right endpoint}. Then, any walk $\sigma$ can be uniquely written as the word recording the arrows $\alpha$ or opposite arrows $\alpha^{-1}$ met on the way from its left endpoint to its right endpoint.

\begin{figure}
    \centering
 \begin{tikzpicture}
\draw (0,0) node(A){$\bullet$};
\draw (1,0) node(B){$\bullet$};
\draw (1,1) node(C){$\bullet$};
\draw[>=stealth, ->, line width=1pt] (A) -- node[midway, below] {$\alpha$} (B);
\draw[>=stealth, ->, line width=1pt] (B) -- node[midway, right] {$\beta_{\psi}$} (C);
\draw[dashed] (1,0.4) arc[start angle=90, end angle=180, radius=0.4cm];
\draw (-1.3,1.3) -- (0,0);
\draw (-0.3,2.3) -- (1,1);
\draw (1,0) -- (2.3,-1.3);

\draw[line width=6pt, red, opacity=0.3] (-1.3,1.3) -- (0,0) -- (1,0) -- (2.3,-1.3);
\draw[line width=6pt, blue, opacity=0.3] (-0.3,2.3) -- (1,1) -- (1,0.2) -- (2.4,-1.2);

\draw (0.8,2) node[blue]{$\sigma'$};
\draw (-0.8,0.2) node[red]{$\sigma$};

\begin{scope}[xshift=5cm, scale=1.2]
\draw (0,0) node(D){$0$};
\draw (1,0) node(E){$\mathbb{K}$};
\draw (1,1) node(F){$\mathbb{K}$};
\draw[>=stealth, ->, line width=1pt] (D) -- node[midway, above] {$0$} (E);
\draw[>=stealth, ->, line width=1pt] (E) -- node[midway, right] {id} (F);
\draw[dashed] (1,0.4) arc[start angle=90, end angle=180, radius=0.4cm];
\draw[dashed] (-1,1) -- (-0.2,0.2);
\draw[dashed] (0,2) -- (0.8,1.2);
\draw[dashed] (1.2,-0.2) -- (2,-1);

\begin{scope}[xshift=4cm]
\draw (0,0) node(G){$\mathbb{K}$};
\draw (1,0) node(H){$\mathbb{K}$};
\draw (1,1) node(I){$0$};
\draw[>=stealth, ->, line width=1pt] (G) -- node[midway, below] {id} (H);
\draw[>=stealth, ->, line width=1pt] (H) -- node[midway, right] {$0$} (I);
\draw[dashed] (1,0.4) arc[start angle=90, end angle=180, radius=0.4cm];
\draw[dashed] (-1,1) -- (-0.2,0.2);
\draw[dashed] (0,2) -- (0.8,1.2);
\draw[dashed] (1.2,-0.2) -- (2,-1);  
\end{scope}

\draw [>=stealth, ->, line width=1pt, bend left=10] (F) to node[midway, above] {$0$} (I);
\draw [>=stealth, ->, line width=1pt, bend left=20] (E) to node[midway, above] {id} (H);
\draw [>=stealth, ->, line width=1pt, bend right=30] (D) to node[midway, above] {$0$} (G);
\end{scope}
 \end{tikzpicture}
    \caption{Illustration of the proof of Lemma \ref{lem:twostringshomortho}, showing on the right a part of the non-zero morphism from $C_{\sigma'}$ to $C_{\sigma}$.}
    \label{fig:proofstring}
\end{figure}
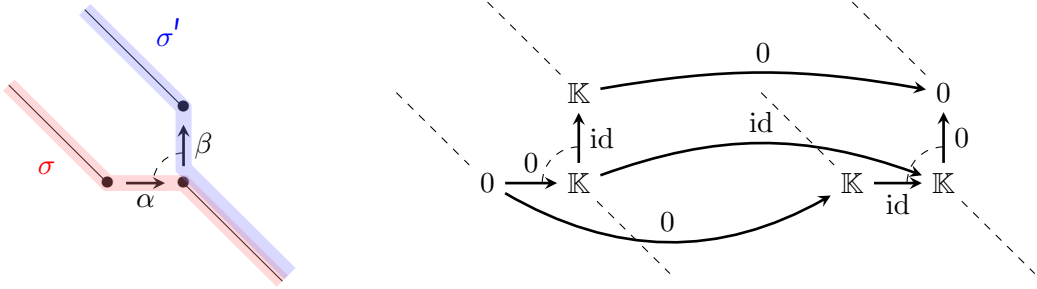

\begin{lemma}
\label{lem:twostringshomortho}
Let $\sigma$ and $\sigma'$ be two different walks on a gentle tree $T$ that share the same right endpoint. Then $C_{\sigma}$ and $C_{\sigma'}$ are not hom-orthogonal.
\end{lemma}

\begin{proof}
The walks $\sigma,\sigma'$, sharing the same right endpoint, can be written $\sigma=\sigma_1 \alpha \gamma$ and $\sigma'=\sigma'_1 \beta^{-1} \gamma$ where $\alpha \neq \beta_{\psi}$. Moreover since the underlying undirected graph of $T$ is a tree, we know that the walks $\sigma$ and $\sigma'$ have disjoint supports. Without loss of generality, we can assume that $\alpha$ is an east step and $\beta_{\psi}$ a north step. Then we have a non-zero morphism from $C_{\sigma'}$ to $C_{\sigma}$, which is the identity map on the vertices of $\gamma$ and the zero map on every other vertex. See Figure \ref{fig:proofstring}, with on the right the crucial part of this non-zero morphism.
\end{proof}

\begin{proposition}[{\cite[Theorem 1.8]{barnard2019minimal}}]
\label{prop:semibrickcanonical}
If $A$ is representation finite, the faces of the canonical join complex are the semibricks.
\end{proposition}

Thus for $A$ representation finite, $\overline{G(\mathrm{Tors}(A))}$ is the graph whose vertices are the bricks and we have an edge between $B$ and $B'$ if they are hom-orthogonal.

\begin{figure}
	\centering
	\begin{subfigure}[t]{.5\textwidth}
		\centering
\begin{tikzpicture}[scale=1.4]
\draw (0,0) node (D){$1$};
\draw (1,1) node (A){$3$};
\draw (1,0) node (B){$2$};
\draw (2,0) node (C){$4$};
\draw[dashed] (1,0.4) arc[start angle=90, end angle=180, radius=0.4cm];
\draw[>=stealth, ->, line width=1pt] (B) -- node[midway, right] {$b$} (A);
\draw[>=stealth, ->, line width=1pt] (B) -- node[midway, below] {$c$} (C);
\draw[>=stealth, ->, line width=1pt] (D) -- node[midway, below] {$a$} (B);
\end{tikzpicture}
		\caption{A gentle tree $T$ with relation $ab=0$.}
		\label{fig:gentletreeexample}
	\end{subfigure}%
	\begin{subfigure}[t]{.5\textwidth}
		\centering
\begin{tikzpicture}[x=0.0223cm,y=-0.0223cm,scale=0.6]
   \begin{scope}[every node/.style={circle,inner sep=1pt}]
    \node (0) at (32.500,130.000) {$\mathrm{mod}\, A$};
    \node (1) at (-70.000,190.000) {$\bullet$};
    \node (2) at (32.500,240.000) {$\bullet$};
    \node (3) at (97.500,240.000) {$\bullet$};
    \node (4) at (-97.500,240.000) {$\bullet$};
    \node (5) at (-22.500,240.000) {$\bullet$};
    \node (6) at (-87.500,350.000) {$\bullet$};
    \node (7) at (-22.500,350.000) {$\bullet$};
    \node (8) at (32.500,350.000) {$\bullet$};
    \node (9) at (-87.500,410.000) {$\bullet$};
    \node (10) at (-60.000,495.000) {$\bullet$};
    \node (11) at (87.500,350.000) {$\bullet$};
    \node (12) at (-5.000,495.000) {$\bullet$};
    \node (13) at (-5.000,550.000) {$\bullet$};
   \end{scope}
   \begin{scope}[every node/.style={fill=white}]
    \path (0) edge[] (1);
    \path (0) edge[] (2);
    \path (0) edge[] (3);
    \path (1) edge[] (4);
    \path (1) edge[] (5);
    \path (4) edge[] (6);
    \path (4) edge[] (7);
    \path (5) edge[] (6);
    \path (5) edge[] (8);
    \path (6) edge[] (9);
    \path (7) edge[] (10);
    \path (2) edge[] (7);
    \path (2) edge[] (11);
    \path (3) edge[] (8);
    \path (3) edge[] (11);
    \path (8) edge[] (12);
    \path (9) edge[] (10);
    \path (9) edge[] (12);
    \path (10) edge[] (13);
    \path (11) edge[] (13);
    \path (12) edge[] (13);
   \end{scope}
  \end{tikzpicture}
		\caption{The lattice of torsion classes of $1 \rightarrow 2 \leftarrow 3$.}
		\label{fig:latticetorsion}
	\end{subfigure}
	\caption{}
	\label{fig:exampletorsionclasses}
\end{figure}

\begin{proposition}
\label{prop:dimgentletree}
Let $T$ be a gentle tree having $n$ vertices. Then $\dim(\mathrm{Tors}(T))=n$.    
\end{proposition}

\begin{proof}
Since a gentle tree is representation finite, Proposition \ref{prop:semibrickcanonical} applies. The category $\mathrm{mod}\, A$ has $n$ simples and they form a semibrick as their supports are disjoint. Thus $\dim(\mathrm{Tors}(T))~\geq~n$. To conclude, let us prove that $\dim(\mathrm{Tors}(T))\leq n$. For this we find a proper coloring of $\overline{G(\mathrm{Tors}(T))}$ with $n$ colors. Since $T$ is a gentle tree, every brick of $T$ is a string representation $C_{\sigma}$, thus it has a right endpoint. Color each brick with the label of its right endpoint. This gives a coloring with $n$ colors of $\overline{G(\mathrm{Tors}(T))}$, which is proper by Lemma \ref{lem:twostringshomortho}. This proves the proposition.
\end{proof}

\begin{corollary}
\label{cor:dimCambrianA}
We have $\dim \big(\mathrm{Camb}(\overrightarrow{\mathbf{A_n}})\big)=n$ for any orientation of $\mathbf{A}_n$.   
\end{corollary}

\begin{proof}
This follows from Proposition \ref{prop:dimgentletree} since $\mathrm{Camb}(\overrightarrow{\mathbf{A_n}})$ is the lattice of torsion classes of an orientation of the path quiver $\mathbf{A}_n$, which is a particular gentle tree having $n$ vertices.    
\end{proof}

\section{Open questions}
\label{sec:openquestions}

We finish by giving some open questions.

In \cite[Figure 9]{barnard2016canonical} an example is given of an extremal congruence uniform lattice of dimension 3 obtained from the doubling of a chain, thus of dimension 1, of an extremal congruence uniform lattice of dimension 2. This is a counterexample to the observation that often $\dim(L[C]) = \dim(L)$ when the interval $C$ satisfies $\dim(C)<\dim(L)$.

\begin{question}
	Let $L$ be a lattice. Under which condition on an interval $C$ of $L$ do we have $\dim(L[C]) = \dim(L)$?   
\end{question}

In Theorem \ref{thm:CUextremalshellable}, we proved the equivalence between extremality and shellability in the case of congruence uniform lattices. For the more general case of semidistributive lattices, we know that extremality implies shellability (Corollary \ref{CorExtremalSD}). What about the converse?

\begin{question}
	Is it true that any semidistributive shellable lattice is extremal?    
\end{question}

Another direction to generalize Theorem \ref{thm:CUextremalshellable} is to look at join-congruence uniform and meet-congruence uniform lattices. We proved in Corollary \ref{cor:mainresultLBUB} that any join-congruence uniform lattice that is shellable is join-extremal. What about the converse:

\begin{question}
	Is it true that any join-congruence uniform lattice that is join-extremal is shellable?
\end{question}

In Section \ref{sec:canonicaljoingraph} we proved that an induced subgraph of a canonical join graph is not always a canonical join graph, even if the initial canonical join graph can be obtained from a semidistributive lattice that is extremal. For similar questions, it would be nice to have an affirmative answer to the following question, as it would restrict the class of lattices we have to look at:

\begin{question}
Let $G$ be the canonical join graph of a semidistributive lattice. Is it true that $G$ can also be obtained as the canonical join graph of an extremal semidistributive lattice? 
\end{question}

Different lattices that we considered in this paper can be obtained as the orientation of the 1-skeleton of a polytope by choosing a generic cost vector (see the precise definition in \cite{hersh2024posets}). In our examples, all the ones coming from a $n$-dimensional polytope were of order dimension $n$. Since in these examples those lattices are all semidistributive we know that we have the lower bound $n$ for their order dimension by Proposition \ref{prop:dimSDbigger}.

\begin{question}
	\label{question:poly}
	If a lattice $L$ is obtained as an orientation of the 1-skeleton of a $n$-dimensional polytope, is it true that $\dim(L)=n$? If it is false,  does it hold with the extra hypothesis that $L$ is semidistributive and/or extremal?
\end{question}

Note that it seems that the face lattice associated to a polytope behaves very differently since for any $n\geq 4$ there exist $n$-dimensional polytopes whose face lattice has arbitrarily large order dimension, and for $n=3$ the face lattice has order dimension less than or equal to 4 \cite{brightwell1993order}.

The lattices of torsion classes are \emph{Hasse-regular} lattices, meaning any element $x$ in their Hasse diagram has the same degree, which is the sum of the number of elements $x$ covers with the number of elements that cover $x$. A lattice is called \emph{$r$-Hasse-regular} if the degree of any element is $r$.

\begin{question}
	Is it true that any semidistributive $r$-Hasse-regular lattice is of dimension $r$? If not, is it true if the lattice is additionally congruence uniform?
\end{question}

\bibliographystyle{alpha}
\bibliography{references}

\end{document}